\documentclass[reqno,11pt]{amsart}
\usepackage[foot]{amsaddr}
\usepackage{amsmath, latexsym, amsfonts, amssymb, amsthm, amscd}
\usepackage{todonotes}
\usepackage{comment}
\usepackage{enumitem}

\title[Critical Lyapunov exponent]{On the Lyapunov exponent for the random field Ising transfer matrix, in the critical case}

\usepackage{xcolor}

\usepackage{graphicx}

\usepackage[capitalise,noabbrev]{cleveref}

\numberwithin{equation}{section}

\newtheorem{theorem}{Theorem}[section]
\newtheorem{lemma}[theorem]{Lemma}
\newtheorem{proposition}[theorem]{Proposition}
\newtheorem{cor}[theorem]{Corollary}
\newtheorem{corollary}[theorem]{Corollary}
\newtheorem{rem}[theorem]{Remark}

\DeclareMathOperator{\p}{\mathbb{P}}
\DeclareMathOperator{\e}{\mathbb{E}}

\def\law{\buildrel {\rm (law)} \over {=}} 

\newcommand{\ind}{\mathbf{1}}

\newcommand{\R}{\mathbb{R}}

\renewcommand{\tilde}{\widetilde}

\newcommand{\cB}{{\ensuremath{\mathcal B}} }

\newcommand{\cN}{{\ensuremath{\mathcal N}} }
\newcommand{\cM}{{\ensuremath{\mathcal M}} }
\newcommand{\cL}{{\ensuremath{\mathcal L}} }
\newcommand{\cT}{{\ensuremath{\mathcal T}} }

\newcommand{\bN}{{\ensuremath{\mathbf N}} }

\DeclareMathSymbol{\leqslant}{\mathalpha}{AMSa}{"36} 
\DeclareMathSymbol{\geqslant}{\mathalpha}{AMSa}{"3E} 
\DeclareMathSymbol{\eset}{\mathalpha}{AMSb}{"3F}     
\renewcommand{\leq}{\;\leqslant\;}                   
\newcommand{\dd}{\text{\rm d}}             

\newcommand{\bbE}{{\ensuremath{\mathbb E}} }

\newcommand{\bbL}{{\ensuremath{\mathbb L}} }

\newcommand{\bbP}{{\ensuremath{\mathbb P}} }

\newcommand{\bbR}{{\ensuremath{\mathbb R}} }

\newcommand{\ga}{\alpha}

\newcommand{\gd}{\delta}
\newcommand{\gep}{\varepsilon}       

\newcommand{\gz}{\zeta}
\newcommand{\gG}{\Gamma}

\newcommand{\gs}{\sigma}

\makeatletter
\def\captionfont@{\footnotesize}
\def\captionheadfont@{\scshape}

\long\def\@makecaption#1#2{%
  \vspace{2mm}
  \setbox\@tempboxa\vbox{\color@setgroup
    \advance\hsize-6pc\noindent
    \captionfont@\captionheadfont@#1\@xp\@ifnotempty\@xp
        {\@cdr#2\@nil}{.\captionfont@\upshape\enspace#2}%
    \unskip\kern-6pc\par
    \global\setbox\@ne\lastbox\color@endgroup}%
  \ifhbox\@ne 
    \setbox\@ne\hbox{\unhbox\@ne\unskip\unskip\unpenalty\unkern}%
  \fi
  \ifdim\wd\@tempboxa=\z@ 
    \setbox\@ne\hbox to\columnwidth{\hss\kern-6pc\box\@ne\hss}%
  \else 
    \setbox\@ne\vbox{\unvbox\@tempboxa\parskip\z@skip
        \noindent\unhbox\@ne\advance\hsize-6pc\par}%
\fi
  \ifnum\@tempcnta<64 
    \addvspace\abovecaptionskip
    \moveright 3pc\box\@ne
  \else 
    \moveright 3pc\box\@ne
    \nobreak
    \vskip\belowcaptionskip
  \fi
\relax
}
\makeatother
\def\writefig#1 #2 #3 {\rlap{\kern #1 truecm
\raise #2 truecm \hbox{#3}}}

\newcommand{\logZ}{\mathtt{z}}
\newcommand{\hup}{{\bf  h}}

\newcommand{\lar}{\vartriangleleft}
\newcommand{\rar}{\vartriangleright}

\newcommand{\cfromQM}{75}

\begin{document}

\author[O.\ Collin, G.\ Giacomin, R.L.\ Greenblatt and Y.\ Hu]{Orph\'ee Collin, Giambattista Giacomin, Rafael L.\ Greenblatt and Yueyun Hu}
\address[OC]{Universit\'e  Paris Cit\'e,  Laboratoire de Probabilit{\'e}s, Statistiques  et Mod\'elisation, UMR 8001,
            F-75205 Paris, France}
            \address[GG, corresponding author]{Universit\`a di Padova,    Dipartimento di Matematica ``Tullio Levi-Civita'', Via Trieste 63, I-35121 Padova, Italy; e-mail address giacomin@math.unipd.it
       }
	    \address[RLG]{Universit\`a di Roma ``Tor Vergata'', Dipartimento di Matematica, Rome, Italy
	    }
    \address[YH]{Universit\'e Paris XIII, LAGA,  Institut Galil\'ee, F-93430 Villetaneuse, France}

   \begin{abstract} 
	   We study the top Lyapunov exponent of a product of random $2 \times 2$ matrices appearing in the analysis of several statistical mechanical models with disorder, extending a previous treatment of the critical case (Giacomin and Greenblatt, ALEA {\bf 19} (2022), 701-728) by significantly weakening the assumptions on the disorder distribution. 
The argument we give completely revisits and improves the previous proof. As a key novelty we 
 build a probability that is close to the Furstenberg probability, i.e.\ the invariant probability of the Markov chain corresponding to the evolution of  the direction of a vector in $\R^2$ under the  action of the random matrices,
	     in terms of the ladder times of a centered  random walk which is  directly related to the   random matrix sequence. We then  show that sharp estimates   on the ladder times (renewal) process  lead to a sharp control on the probability measure we build and, in turn, to the control of its distance from the Furstenberg probability.
	    \bigskip

\noindent  \emph{AMS  subject classification (2010 MSC)}:
60K37,  
82B44, 
60K35, 
37H15  	

    \end{abstract}
  
    \date{\today} 
   
\maketitle



\section{Introduction}
\label{sec:setup}

\subsection{The model we consider}
Let
\begin{equation}
\label{eq:keyM}
M\, =\,M(\gep, Z)\, :=\, 
\begin{pmatrix}
1& \gep\\ 
\gep Z  & Z
\end{pmatrix}\, , 
\end{equation}
with $\vert \gep \vert < 1$ and $Z \in (0, \infty)$. 
Consider now an IID sequence of positive random variables $(Z_j)_{j=1,2, \ldots}$ with $\bbE[\vert \log Z_1\vert ] < \infty$ and let us assume that the distribution of $Z$ is not trivial. 
The top Lyapunov exponent of the product of the random matrices $(M(\gep, Z_j))_{j=1,2, \ldots}$ is
\begin{equation}
\label{eq:L_Z}
	\cL_Z(\gep)\, :=\, \lim_{n \to \infty} \frac 1 n \bbE \log \left \Vert M(\gep , Z_1) M(\gep , Z_2) \cdots  M(\gep , Z_n)\right \Vert\, ,
\end{equation}
where $\Vert \cdot \Vert$ is an arbitrary matrix norm. $\cL_Z(\gep)$ is well defined and independent of the matrix norm. Existence of the limit  is a straightforward consequence of 
a well-known property of subadditive sequences (Fekete's Lemma) that can be applied  if $\Vert \cdot \Vert$ is  a submultiplicative norm.
The fact that
all matrix norms are equivalent then yields that $\cL_Z(\gep)$ does not depend on the choice of the norm. This result holds  for every $\gep\in \bbR$, however it will be soon important to focus on invertible matrices, i.e.\ $\vert \gep \vert \neq 1$ 
and we are interested in the behavior of $\cL_Z(\gep)$ when $\gep $ approaches 0, that is why  we restrict to $\vert \gep \vert < 1$. 

Let us remark that if $D$ is the diagonal matrix with $(1,-1)$ on the diagonal, we have $D=D^{-1}$ and 
$M(-\gep,Z)= D M(\gep, Z) D$, which implies  $\cL_Z(\gep)=\cL_Z(-\gep)$ and so without loss of generality we may consider only $\gep \in [0,1)$. 
Moreover 
\smallskip 

\begin{enumerate}[leftmargin=0.7 cm]
\item 
 it is straightforward  
to verify that the conditions of Furstenberg Theorem  (e.g., \cite[Th.~4.1, p.~30]{cf:BL})
are verified for $\gep \in (0,1)$ for   the sequence of unit determinant matrices $(M(\gep, Z_j)/\sqrt{(1-\gep^2) Z_j})_{j=1,2, \ldots}$
This has a number of consequences, in particular $\gep \mapsto \cL_Z(\gep)$ is $C^0$ and that there is a formula for $\cL_Z(\gep)$ that we give and (crucially) exploit below.
\item For $\gep \in (0,1)$ the problem we consider falls into the cone contraction theory of \cite{cf:Ruelle} and $\gep \mapsto  \cL_Z(\gep)$ is real analytic.
\end{enumerate}
\smallskip

The case $\gep=0$ is trivial in the sense that we explicitly compute $\cL_Z(0)= \max(0, \bbE[ \log Z])$. But the case $\gep=0$ fails to satisfy the hypotheses
of Furstenberg Theorem: in fact the irreducibility property   \cite[Th.~4.1]{cf:BL} fails because $M(0, Z)v\propto v$ both for 
$v=(1,0)^t$ and for $v=(0,1)^t$. So $\gep=0$ is a singular point in the sense that Furstenberg Theorem cannot be applied when $\gep=0$ and this motivates, at least from
the mathematical viewpoint, the study of $\cL_Z(\gep)$ as $\gep \to 0$: we will soon see that the motivation goes well beyond this. 

We notice that if furthermore $\bbE[ \log Z]= 0$, then not only the matrix is diagonal and fails the irreducibility property, but the exponential growth rate of the  terms on the diagonal 
is the same (zero in both cases), that suggests that this case is particularly singular (lack of hyperbolicity). The  $\bbE[ \log Z]= 0$ case is naturally referred as critical: we will come back to discuss and confirm this aspect in Sec.~\ref{sec:discussion}.

A change of perspective on the problem is notationally useful, particularly in the critical case: we set 
\begin{equation}
\gG \, :=\, -\log \gep\ \ \text{ and } \ \ \ \logZ\,:=\, \log Z\,,
\end{equation}
so 
  the original matrix $M$ of \eqref{eq:keyM}
becomes
 \begin{equation}
\label{eq:matrix2}
\begin{pmatrix}
1& \exp(-\gG )\\ 
\exp(-\gG +\logZ)  & \exp(\logZ)
\end{pmatrix}\, ,
\end{equation}
and we set 
\begin{equation}
\cL(\gG)\, :=\, \cL_Z(\gep)\, .
\end{equation}
In analogy with before, $\logZ_j= \log Z_j$, so $(\logZ_j)_{j=1,2, \ldots}$ is an IID sequence. With this notation the critical case means
\begin{equation}
\label{eq:critical}
\logZ\in \bbL^1
\ \ \text{ and } \ \ \bbE[\logZ]\,=\, \int_\bbR x \zeta(\dd x)\, =\, 0\,,
\end{equation}
where $\zeta $ is the law of $\logZ$. 

\subsection{Our main result}
\label{sec:mainresult}

In \cite{cf:GG22} two of the authors proved the following result: 
\medskip

\begin{theorem}
\label{th:main-old}
Choose $\zeta$ that satisfies 
 \eqref{eq:critical} and such that
 \begin{description}
		\item[{(h-1)}]  $\zeta$ has a density and its density is  H\"older continuous uniformly over $\bbR$;
\item[{(h-2)}] there exists $c>0$ such that $\int_\bbR \exp( c \vert x\vert) \zeta(\dd x) < \infty$.
\end{description}
Then there exists $\kappa_1>0$, $\kappa_2 \in \bbR$  and $\gd >0$ such that
\begin{equation}
\label{eq:main-old}
\cL(\gG) \stackrel{\gG \to \infty}= \frac{\kappa_1}{\gG+ \kappa_2} + O \left( \exp\left( - \gd \gG \right)\right)\,.
\end{equation}
\end{theorem}

\medskip

Here we prove the following result: for $n=1,2, \ldots $ denote by $\zeta^{*n}$ the law of $\logZ_1+ \ldots + \logZ_n$.
\medskip

\begin{theorem}
\label{th:main}
Choose $\zeta$ that satisfies 
 \eqref{eq:critical} and such that
 \begin{description}
		\item[{(H-1)}]  there exists $n_0\ge 1$ such that $\zeta^{*n_0}$ has a density;
	\end{description}
and either
 \begin{description}
		\item[{(H-2)}] there exists $\xi > 5$ such that $\int_\bbR \vert x\vert ^\xi \zeta (\dd x) < \infty$;
	\end{description}
	or
	\begin{description}
		\item[{(H-2$^\prime$)}] there exists $c>0$ such that $\int_\bbR \exp\left(c \vert x\vert\right) \zeta (\dd x) < \infty$.
	\end{description}
Then there exists $\kappa_1>0$ and $\kappa_2 \in \bbR$  such that 
\begin{equation}
\label{eq:main}
\cL(\gG) \stackrel{\gG \to \infty}= \frac{\kappa_1}{\gG+ \kappa_2} + \textsc{R}(\gG)  \, .
\end{equation}
where for $\gG\to \infty$
\begin{equation}
\textsc{R}(\gG)\, = \begin{cases}
O\left( \gG^{-(\xi-4)} \right) & \text{ under  {\bf (H-2)}},
\\
O\left(\exp\left(- \gd \gG\right) \right) & \text{ under {\bf (H-2$^\prime$)}},
\end{cases}
\end{equation}
for a suitable choice of $\gd>0$.
\end{theorem}
\medskip

Let us stress the elementary fact that
\begin{equation}
\frac{\kappa_1}{\gG+ \kappa_2}\, =\, \kappa_1 \sum_{j=1}^\infty (-1)^{j-1} \kappa_2^{j-1} \gG^{-j}\, , 
\end{equation}
so that \eqref{eq:main-old} provides a development of $\cL(\gG)$ to all orders. 
This is also the case for \eqref{eq:main}  if   {\bf (H-2)} holds for $\xi$ arbitrarily large, 
but of course this only gives  
 $\textsc{R}(\gG)=O(\gG^{-C})$ for every $C>0$. An exponential control on the rest is  recovered under hypothesis {\bf (H-2$^\prime$)}$=${\bf(h-2)}.

When $\xi < \infty$, i.e.\ $\xi \in (5, \infty)$ then \eqref{eq:main} is equivalent to 
\begin{equation}
\label{eq:main-1}
\cL(\gG) \stackrel{\gG \to \infty}=\kappa_1 \sum_{j=1}^{\lfloor \xi \rfloor-4} (-1)^{j-1} \kappa_2^{j-1} \gG^{-j} + O\left( \gG^{-( \xi-4)}  \right) \, .
\end{equation}
with $\lfloor \cdot \rfloor$  the lower integer part of $\cdot$. 
Therefore for $\xi \in (5,6]$ we only have $\cL(\gG) = \kappa_1  \gG^{-1} + O(\gG^{-1-\epsilon})$, with
$\epsilon=\xi-5$, and $\kappa_2$ enters the game only for $\xi >6$.

\medskip

An important point of these results is that we do have expressions for $\kappa_1$ and $\kappa_2$ -- they are  in \eqref{eq:kappa_1_formula-0} and 
\eqref{eq:kappa2} --
and they are not simple expressions because they depend on the invariant measure of a suitable null recurrent Markov Chain (MC).
The dependence on invariant measures of MC's is not surprising because this is in the nature of  the Furstenberg expression for $\cL(\gG)$ (see \eqref{eq:Lyap-k}
below). 
However one of us proved \cite{cf:Orphee}, by  using the statistical mechanics origin of the question and avoiding the  Furstenberg expression for
 $\cL(\gG)$, that 
$\kappa_1$ is equal to the variance of the external magnetic field
and this result has been established also in \cite{cf:DMSDS-B} with an approach that stays in the realm of Furstenberg theory.
In the variables we use here this means that $4\kappa_1$ is the 
variance of
 $\logZ$, i.e.\ $\int_ \bbR x^2 \zeta(\dd x)$. This very intriguing point, proven  in \cite{cf:Orphee} 
 assuming only $\logZ \in \bbL^\xi$, $\xi>(3+\sqrt{5})/2=2.618\ldots$ and in 
 \cite{cf:DMSDS-B} assuming $\logZ \in \bbL^\infty$ (with a control on the error bound),
  is picked up again in Sec.~\ref{sec:discussion}. Moreover in App.~\ref{sec:kappa1=var} we show that 
  the  estimates we obtain can be used to provide  an alternative proof 
   of 
 the result in \cite{cf:Orphee,cf:DMSDS-B}. Of course with respect to \cite{cf:Orphee,cf:DMSDS-B} we do assume {\bf (H-1)}, so, in particular,
 the proof of $4\kappa_1=\int_ \bbR x^2 \zeta(\dd x)$ in App.~\ref{sec:kappa1=var}  is obtained under stronger 
 assumptions than in \cite{cf:Orphee}, 
 but we underline that assuming {\bf (H-1)} we   obtain a substantially sharper asymptotic statement and the
 emergence of another universal behavior as we discuss at the end of Sec.~\ref{sec:discussion}.
 
 {Finally, in view of the results in \cite{cf:Orphee} and of the fact that we believe  that  our technical estimates are close to being optimal,  the moment condition $\xi>5$ appears to be  a limitation of the method we use: for example the method does require a control on the invariant measure that involves the precise value of the constant $\kappa_2$ even when $\xi \in (5, 6]$, and  $\kappa_2$ is irrelevant for  the final result for $\xi \in (5, 6]$. On the other hand, it is definitely an open issue whether the error estimate on $\textsc{R}(\gG)$ is optimal (for example) in the following sense: do there exist random variables $\logZ\not\in \bbL^5$ such that    $\sup_\gG \vert \gG^2(\cL(\gG)- \kappa_1/ \gG)\vert=\infty$? Or is it even the generic  situation  in this case? Going farther in this direction and beyond the framework of this work,  it may be possible to  extend the result in  \cite{cf:Orphee} down to  $\logZ\not\in \bbL^2$, but a scaling different from $1/ \gG$ is expected for $\cL(\gG)$  if $\logZ\not\in \bbL^2$. 
 }

\subsection{Origin of the question and discussion of the literature}
\label{sec:discussion}

The key matrix $M$ in \eqref{eq:keyM} 
is proportional to the transfer matrix of the one dimensional Ising model
(see e.g. \cite[Ch.~3]{cf:Luck} or \cite{cf:DH83}):
\begin{equation}
\label{eq:Tmatrix}
 \begin{pmatrix} e^{J+h} & e^{-J+h} \\
e^{-J-h} &  e^{J-h} 
\end{pmatrix}\, =\, e^{J+h} \, \begin{pmatrix} 1 & e^{-2J} \\
e^{-2J-2h} &  e^{-2h} 
\end{pmatrix}\, =\,  e^{J+h}  M\left( e^{-2J}, e^{-2h} \right)\, , 
\end{equation} 
where $h$ is the external field and $J$ the interaction. Explicitly, the partition function of the Ising model of length $n$ with periodic boundary conditions and  with inhomogeneous external field $h_1,h_2,\ldots, h_n$ is equal to $\exp(JN + \sum_{j=1}^n h_j)$ times the trace of  $M\left( e^{-2J}, e^{-2h_1} \right)\cdots M\left( e^{-2J}, e^{-2h_n} \right)$, where we consider $(h_n)_{n\ge 1}$ as an IID sequence of random variables, distributed as $h$. Hence the Lyapunov exponent we consider is the free energy density of the   
random external field  one dimensional Ising model minus $J$.
The same transfer matrix product and Lyapunov exponent enter the exact expression for the free energy density of the two dimensional Ising model with zero external field and \emph{columnar disorder} \cite{cf:CGG,cf:MW1,cf:MWbook}.  
The seminal works in the analysis of this matrix in statistical mechanics of disordered systems are 
\cite{cf:DH83} and  \cite{cf:MW1};
\begin{enumerate}[leftmargin=0.7 cm]
\item In the first work, McCoy and Wu provide a solution to a continuum model arising in the weak disorder limit 
of the matrix product model we consider
(a mathematical approach to this limit can be found in \cite{cf:CGG} and solvable continuum limits for a much greater class of models may be found in \cite{cf:CLTT2013}). The Lyapunov exponent of the continuum model is explicitly expressed in terms of Bessel functions,   and one easily obtains an arbitrarily precise analysis, notably in the $J \to \infty$ limit: the $e^{-2J}$ appearing in the matrix in the $2d$ Ising model considered in \cite{cf:MW1} is not an interaction term, but rather a \emph{phase parameter} and key to understanding the phase transition is the singularity arising near phase zero, which explains the link with the $J \to \infty$ limit of the $1d$ case.
\item In the second work, Derrida and Hilhorst  provide an analysis of the $J \to \infty$ limit of the Lyapunov exponent in the case where the expectation of $h$ is nonzero, targeting directly the one dimensional case, but mentioning several other cases in which this limit is relevant, see 
 \cite{cf:DH83}, but also the monograph \cite{CPV}. 
The case of centered $h$ has been first taken up in \cite{cf:NL86}. 
\end{enumerate}

For the physical relevance of the analysis of the $J\to \infty$ and of why the $\bbE[h]=0$ case  is critical we refer in particular to \cite{cf:CGG,CPV,cf:DH83,cf:MWbook}. Here we insist on the fact that our main result \eqref{eq:main} arises in the exactly solvable McCoy-Wu continuum model, with an exponentially small error term $R(\gG)$ (in fact, the sharp behavior of this term is known in this case): in the continuum model the noise, or \emph{disorder}, is Brownian, hence with Gaussian tails. 
A priori it is not obvious to which extent the result obtained in the continuum limit corresponds to the behavior of the original matrix product:
in the off-critical case, i.e.\ $\bbE[h] \neq 0$, the leading behavior of $\cL(\gG)$ \cite{cf:DH83}
corresponds to the leading order of the continuum case \cite{cf:CGG}, and this has been mathematically established in  \cite{cf:GGG17} when the disorder is compactly supported and has a $C^1$ density. 
In \cite{cf:NL86} it is claimed that for the $\bbE[h]=0$ a sharp version of \eqref{eq:main},
i.e.\ a development to all orders in powers of $1/ \gG$, holds for a special compactly supported disorder law for which some exact computations are possible, but the analysis is not rigorous and no control on the error term $R(\gG)$ is given. 
This naturally sets forth the issue of understanding, at a mathematical level and for matrix products,  the validity of the result emerging in the continuum model. 

A step in this direction has been performed in
 \cite{cf:GG22} where   \Cref{th:main-old} is established: \eqref{eq:main} holds with an exponential control on $R(\gG)$
under the assumption that the distribution of the disorder
has exponential tails and its density is uniformly H\"older continuous.  
 
Here we greatly generalize the result in  \cite{cf:GG22}: in particular, we establish the result without assuming regularity of the density and, above all, we work with disorder power law tails. In this case the error term $R(\gG)$ has a power law behavior and therefore the free energy development holds only up to a finite order, see
 \eqref{eq:main-1}.
 
Like in  \cite{cf:GGG17,cf:GG22}, our method is based on constructing an explicit measure, the Derrida-Hilhorst probability proposed in \cite{cf:DH83} -- we introduce it in \eqref{eq:DHprobability} and we call it  $\gamma_\gG$ -- that is expected to be close to the Furstenberg measure (that we call  $\nu_\gG$), which is the unique invariant probability associated to the action of the matrix product in the projective space \cite{cf:BL}.
The Lyapunov exponent can be expressed in a straightforward way in terms $\nu_\gG$, as in \cref{eq:Lyap-k}, and insofar as $\gamma_\Gamma$ is close to the true stationary measure it can be used in this expression to give an approximation for the Lyapunov exponent.
 The construction of the Derrida-Hilhorst probability and showing that it is indeed close to the Furstenberg measure is the mathematical challenge. 
 The Derrida-Hilhorst probability is built by using the invariant measures  of two associated 
  Markov chains: when $\bbE[h]\neq0$ one chain is positive recurrent and the other is transient, while in the $\bbE[h]=0$ both chains are null recurrent, so  the invariant measures are not finite measures. 
 These Markov chains constitute a
   classical and much studied problem in probability and statistics, the \emph{random difference equation} ``$X=AX+B$'', see for example \cite{cf:BDMbook,cf:KestenAM}.   In this context, the case $\bbE[h]=0$ is naturally considered critical in the sense of random dynamical systems \cite{cf:BBE97,cf:BDMbook}. A contribution of our work is also providing a very precise solution to this problem in a special case.
 
   The main novelty of the current work, presented in \cref{sec:stationary}, is in  the analysis of the invariant measure of the null recurrent Markov chains 
 $Q_{n+1}= Z_{n+1} (Q_n+1)$, with $Q_0\ge0$ and $Z_{n+1}=\exp(-2h_{n+1})$, with $\bbE[h]=0$: the process $Y$ introduced in Sec.~\ref{sec:Y-MC} is
 $\log Q$. In the critical case  $Y=\log Q$ becomes a centered random walk -- the increments are $(-2h_j)$ -- with a force term that biases the walk against entering the negative half-line. The key to this treatment is exploiting renewal estimates to obtain  a precise control
  on the ladder process of the centered random walk and use it to control the invariant measure of the Markov chain $Y$ (or, equivalently, $Q$).
 
 Once the invariant measure of the $Y$ process is under control the Derrida-Hilhorst probability $\gamma_\gG$ can be built and we are left with showing that $\gamma_\gG$ is close to the Furstenberg probability $\nu_\gG$. This part follows the two step procedure used in \cite{cf:GG22}, and before that in \cite{cf:GGG17} in the noncritical case. Both steps are improved with respect to
 \cite{cf:GG22}:  
  \begin{enumerate}[leftmargin=0.7 cm]
 \item One step consists in establishing a \emph{residual} contraction property of the action of the Furstenberg  Markov chain in the critical case: at criticality the spectral gap vanishes with the strength of the interaction $\gG=2J$ and  in \cite{cf:GG22} an estimate $O(\log \gG / \gG^2)$ is established via diffusion approximation. Here we exploit a very different, more microscopic argument that yields $O( 1/\gG^2)$ and it is a probabilistic version of an argument  due to Q.~Moulard \cite{cf:QM}.  
 \item The other step consists in showing, in a quantitative way and in a precise sense, that the one step action of the Furstenberg  Markov chain does not modify much the Derrida-Hilhorst probability $\gamma_\gG$: the point here is generalizing the one-step bound developed in \cite{cf:GG22} from exponential to 
  power law disorder tail. 
 \end{enumerate}
 
The result in \cite{cf:Orphee}, i.e.\ that $\cL(\gG)\sim \mathrm{var}(\logZ)/ (4\gG)$ for $\logZ \in \bbL^\xi$ with $\xi>(3+\sqrt{5})/2$,
shows that our moment condition $\xi>5$ is not optimal. An improvement in this direction  appears to be rather challenging, since it 
seems to require new ideas about how to construct the invariant law of the $Y$ process, or a  new approach, i.e.\ not involving the Derrida-Hilhorst construction,  in proving properties of the invariant probability, or Furstenberg measure, $\nu_\gG$. We signal however that our approach does provide additional information, in particular a by-product is that if \eqref{eq:critical} and {\bf (H-1)} hold along with $\logZ \in \bbL^\xi $ for $\xi >4$ we have
\begin{equation}
\label{eq:W-1bound}
\mathrm{d}_{W,1} \left( \nu_\gG, \gamma_\gG\right)\, =\, O \left( \gG^{-(\xi-4)} \right)\, ,
 \end{equation}
 where $\mathrm{d}_{W,1} $ is the Wasserstein$-1$ distance: \eqref{eq:W-1bound} is a direct consequence of Corollary~\ref{th:bound-cor}
 and \eqref{eq:onestep} of Proposition~\ref{thm:onestep}. 
 {This is substantially stronger than the control on $\nu_\gG$
 than what is obtained in \cite{cf:DMSDS-B} (assuming $\logZ$ compactly supported, but not assuming hypothesis {\bf (H-1)}). 
 In \cite{cf:DMSDS-B}, see also \cite{cf:footprint}, the first order control on the Lyapunov exponent is extracted exploiting only  the weaker control on $\nu_\gG$ by a nice manipulation trick on the Lyapunov exponent formula.}
 
 We remark also that the fact that $\kappa_1= \mathrm{var}(\logZ)/4$ is strongly suggested by Fisher's Renormalization Group  (RG) approach \cite{cf:F95,cf:IM}. In \cite{cf:CGH1,cf:CGH2} we have  proven  results that confirm 
 the application of the Fisher RG to the critical random field $1d$ Ising model in
  \cite{cf:FLDM01}. The main results in  \cite{cf:CGH1,cf:CGH2} aim at understanding how the 
critical random field  $1d$ Ising configurations look like when $\gG=2J \to \infty$. In particular, in \cite{cf:CGH2} the results are shown assuming only that $\logZ\in \bbL^2$ (i.e.\ $\xi=2$) centered and nontrivial. Moreover, as already discussed at length in  
 \cite{cf:CGH1,cf:CGH2},  the application of Fisher RG in \cite{cf:FLDM01} does not yield the correct behavior of $\cL(\gG)$ beyond the leading order. This is probably  not surprising because the subleading correction is not expected to be universal. And even restricting to the $1d$ Ising model, we see that the subleading corrections  involve $\kappa_2$, which appears to depend on  details 
 of the law of $\logZ$ (even if this is only based on numerical evidence at this stage). 
 Nonetheless we stress that our result  sets forth another, somewhat surprising, form of universality: given the variance of $\logZ$, the subleading corrections  to all order in powers of  $1/ \gG$ depend only on the constant $\kappa_2$.

\subsection{Organisation of the rest of the paper}
In section~\ref{sec:Furstenberg} we recall the necessary tools from Furstenberg theory and we introduce the MC whose invariant measure
 is used to build the probability $\gamma_\gG$ that will be shown  to be a good approximation of the Furstenberg probability $\nu_\gG$. In 
 Sec.~\ref{sec:proofofmaintheorem} we state three results -- \Cref{thm:stationary_asymptotics}, \Cref{thm:contractive} and 
 \Cref{thm:onestep} -- that, combined, yield a proof of
  Theorem~\ref{th:main}. The next three sections are devoted to the proof of these three results.

\section{Tools from   Furstenberg theory} 
\label{sec:Furstenberg}

\subsection{The  Furstenberg expression for $\cL(\gG)$}
Key to the Furstenberg theory is understanding the action of our random matrices \eqref{eq:matrix2} on the projective space which in our case 
is simply one dimensional because $P(\bbR^2)\cong \bbR$.
If we parametrize $P(\bbR^2)\cong \bbR$ with the coordinates $(1, \exp(x))^{\mathtt{T}}$ we readily see that
the action of the matrix \eqref{eq:matrix2} is 
\begin{equation}
\label{eq:map}
x \, \mapsto \, z+ \log \left(
\frac{e^{-\gG }+e^x}{1+ e^{x-\gG }}
\right)\, =:\, z+ h_\gG (x)\, .
\end{equation}
We observe that $h_\gG (\cdot)$ is odd and 
$x \mapsto x- h_\gG (x)$, which  is also odd,  is small if $x \in (-\gG ,\gG )$ and far from the boundary points $\pm \gG $:
\begin{equation}
\label{eq:almost-id}
x-h_\gG (x)\, =\,  \log \left(\frac{1+ e^{x-\gG }}
{1+e^{-x-\gG }}
\right)\, 
=\,  O\left( e^{-(\gG -x)}\right)+ O\left( e^{-(\gG +x)}\right)\, ,
\end{equation}
so $h_\gG (x)$ is \emph{very close} to $x$ on an interval that approaches $\bbR$ in the limit $\gG  \to \infty$ (Fig.~\ref{fig:1}).

\medskip

\begin{figure}[h]
\centering
\includegraphics[width=15.5 cm]{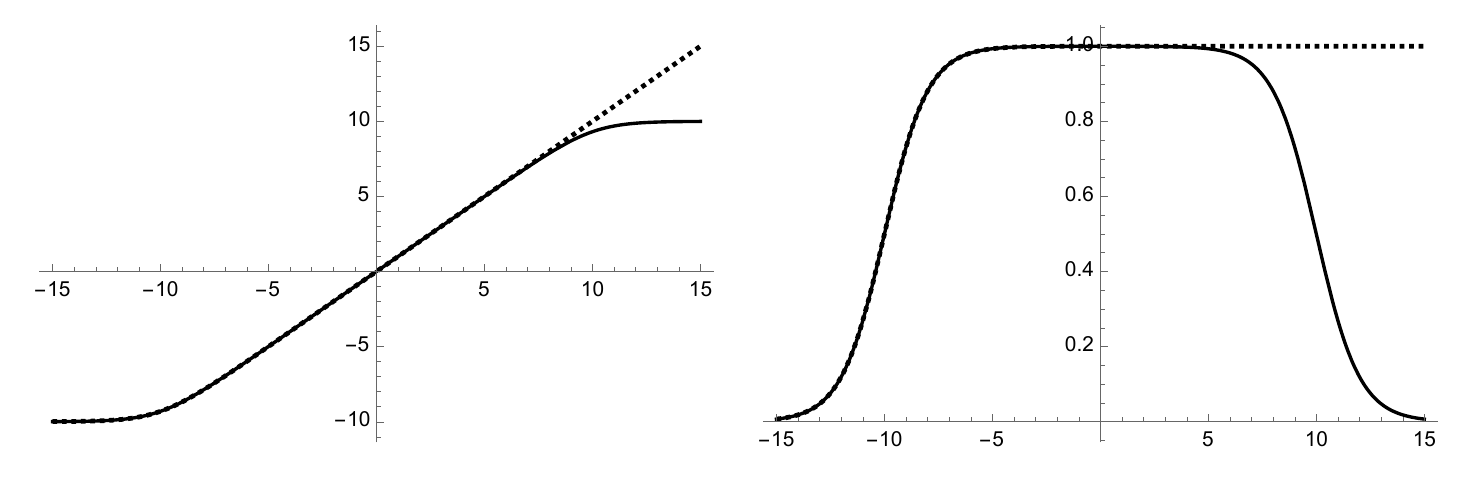}
\caption{\label{fig:1} For $\gG =10$: on the left the plot of $h_\gG (\cdot)$ (solid line) and of $x\mapsto \hup_\gG(x)= h(x+\gG )-\gG $ (dotted line); on the right the derivative of the same functions. }
\end{figure}

Denote by $X=(X_n)_{n=0,1, \ldots}$ the MC generated by 
the map \eqref{eq:map}, that is 
\begin{equation}
\label{eq:XMC}
X_{n+1}\, =\,  \logZ_{n+1} + h_\gG \left(X_n \right)\, .
\end{equation}
Since the image of $h_\gG $ is $(-\gG ,\gG )$ the process $(X_n)$ \emph{hardly leaves} $(-\gG ,\gG )$: in fact we may remark that, no matter what is the value taken by  $X_n$, $\vert X_{n+1}\vert \le \gG + \vert  \logZ_{n+1} \vert$. Moreover
 when the process is in $(-\gG ,\gG )$ and far from the boundaries it is \emph{close to being} a random walk with centered increments. 
 
The MC $X$ has a unique invariant probability $\nu_\gG $:  this can be established just under the assumption that $\logZ$ is not trivial either 
by checking the assumptions of the Furstenberg Theorem \cite{cf:BL} or by remarking the contractive property of this MC 
(this is detailed in \cite[Lemma~2.2]{cf:CGH2})
and we have that
\begin{equation}
\label{eq:Lyap-k}
 \cL(\gG )
\,=\, \int _\bbR \log(1+ \exp(-\gG  + x)) \nu_\gG (\dd x) \, .
\end{equation} 

\begin{rem} \label{rem:sym1}
 Formula \eqref{eq:Lyap-k}  is what follows by specializing the Furstenberg formula for the top Lyapunov exponent \cite[Th.~3.6]{cf:BL} to our context. More precisely, this is one of the possible equivalent formulas that one finds applying Furstenberg formula with different matrix norms, or even semi-norms
 The fact that these are equivalent is a byproduct of the theory that for positive entry matrices, the exponential rate of growth of every entry of the matrix product is given by the top Lyapunov exponent. In particular we have also
 \begin{equation}
 \label{eq:Lyap-k2}
 \cL(\gG )
\,=\, \int _\bbR \log(1+ \exp(-\gG  - x)) \nu_\gG (\dd x) \, , 
\end{equation}  
that a priori may look surprising, but from \eqref{eq:XMC} one obtains that $\int_ \bbR \vert x \vert \nu_\gG(\dd x) \le \gG + \bbE[ \vert \logZ\vert] < \infty$ and  $\int_\bbR (x - h_\gG(x))\nu_\gG(\dd x)=0$,
 and this identity offers a direct way to see  that the two expressions for the Lyapunov exponent in 
\eqref{eq:Lyap-k} and \eqref{eq:Lyap-k2} coincide.
 \end{rem}

\subsection{Building the framework to estimate $\cL(\gG)$} 
In general, if $A$ is a real-valued random variable, we use the notations $G_A(x)=\bbP(A>x)$ and $F_A(x)=1-G_A(x)$: $G_A$ is the risk (distribution) function of $A$ and $F_A$ is the cumulative (distribution) function of $A$. Likewise we speak of risk and cumulative functions of a probability on $\bbR$ and we use this  notation and terminology  also for finite  measures $\mu$ on $\R$: 
$G_\mu(x)= \mu((x, \infty))$ and 
$F_\mu(x)= \mu((-\infty, x])$. 

\medskip
\begin{rem}
\label{rem:infinite}
Later on, we will work also with nonnormalizable ($\gs$-finite) measures, notably with measures $\mu$ such that $\mu((-\infty, x])< \infty$ for every $x$, and the notation $F_\mu(x)$ will be used in this case as well. However, unless specified otherwise,  $\mu$ is a finite measure. 
\end{rem}
\medskip

For $G=G_\mu$, we define  $T_\Gamma G$ by 
\begin{equation}
\label{eq:TGdef}
 T_\gG G(x)\,: =\, G(-\infty)G_{\logZ}((\gG +x)_-) + \int_{(-\gG +x,\gG +x)} G\left( h_\gG ^{-1}(x-z)\right) \gz( \dd z)
\,,  \quad \forall x\in \R\, ,
\end{equation}
where $G_{\logZ}(y_-):= \lim_{x\uparrow y} G_{\logZ}(x)$ for  $y\in \R$ and as before, $\zeta$ denotes the law of $\logZ$. Similarly for $F=F_\mu$, we define \begin{equation}
\label{eq:FXT}
 T_\gG F(x)\, :=\, F(\infty)  F_{\logZ}(-\gG +x) +\int_{(-\gG +x,\gG +x)} F \left( h_\gG ^{-1}(x-z)\right) \gz(\dd z)
\, , \quad \forall x\in \R\,  .
\end{equation}

\medskip

\begin{lemma} 
We have for $n=0,1, \ldots$, 
\begin{equation}
\label{eq:GXTn} G_{X_{n+1}}= 
T_\gG G_{X_n}   
\,  ,
\end{equation}
and
\begin{equation}
\label{eq:FXTn} 
F_{X_{n+1}} = T_\gG F_{X_n} 
\,  .
\end{equation}
\end{lemma}

\medskip

\noindent
\emph{Proof.} By using that $h_\gG (\cdot)$ is an increasing bijection from $\bbR$ to $(-\gG ,\gG )$ we see that
\begin{equation}
\label{eq:12}
\begin{split}
G_{X_{n+1}}(x)\, &=\, \bbP\left( \logZ_{n+1} +h_\gG \left( X_n\right) >x \right) 
\\
&=\, \bbP \left( x-  \logZ_{n+1}\le -\gG \right)+\bbP\left(h_\gG \left( X_n\right) >x-  \logZ_{n+1} ; \, \vert x-  \logZ_{n+1}\vert <\gG  \right) 
\\
&=\,  \bbP \left( \logZ \ge x+\gG \right)+\bbP\left(X_n >h_\gG ^{-1}\left(x-  \logZ_{n+1}\right) ; \, \vert x-  \logZ_{n+1}\vert <\gG  \right)
\\
&=\,  G_{\logZ}\left((\gG +x)_-\right) + \int_{(-\gG +x,\gG +x)} G_{X_n}\left( h_\gG ^{-1}(x-z)\right) \gz(\dd z) \, .
\end{split}
\end{equation}
The proof for \eqref{eq:FXTn} is analogous.
\qed

Note that although explicitly $h_\Gamma^{-1}(y) =\log\tfrac{e^y-e^{-\Gamma}}{1-e^{y-\Gamma}}$, this precise form is not particularly useful for what follows.

\medskip

The map $T_\gG$ can then be extended by linearity to act on $G$ defined by 
$G(x)=G_{\nu_+}(x)- G_{\nu_-}(x)$, $\nu_\pm$ finite measures (so $\vert G(-\infty)\vert < \infty$):
\begin{equation}
\label{eq:GXT}
 T_\gG G(x)\, =\, G(-\infty)G_{\logZ}((\gG +x)_-) + \int_{(-\gG +x,\gG +x)} G\left( h_\gG ^{-1}(x-z)\right) \gz( \dd z)
\, .
\end{equation}
In the applications we consider $\nu$ is  the difference of two probability measures, so $G(-\infty)=0$ and
$T_\gG$ reduces to $T_{\gG,0}$, with 
\begin{equation}
	T_{\gG,0} G (x)\, :=\, 
	\int_{(-\gG +x,\gG +x)} G\left( h_\gG ^{-1}(x-z)\right) \gz( \dd z)\, , \  \quad \forall x\in \R\, ,
	\label{eq:TG0_def}
\end{equation}
which is well defined  also in the slightly different set-up of  $G\in \bbL^1$. In fact, by the Fubini-Tonelli Theorem and  
since $h_\gG'$ takes values in $(0,1)$ we have
\begin{equation}
\label{eq:Tbound0}
\begin{split}
	\left \Vert T_{\gG,0} G \right \Vert_1\, &\le \, \int_\bbR
	\left(\int_{(-\gG +x,\gG +x)} \left \vert G\left( h_\gG ^{-1}(x-z)\right)\right\vert \gz( \dd z)\right) \dd x
	\\
	&=\, \int_ \bbR \left( \int_\bbR 
	 \left \vert G\left( y\right)\right\vert  h_\gG' (y)\dd y
	\right)  \gz( \dd z)\, \le \, \Vert G \Vert_1
	 \, .
\end{split}	
\end{equation}

\smallskip

\subsection{Looking from the edge: the reduced chain $Y$}
\label{sec:Y-MC}
If we sit on $-\gG $, that is if we make it our new origin, in the limit as $\gG \to \infty$ the MC becomes
\begin{equation}
\label{eq:iterY}
Y_{n+1}\, =\,  \logZ_{n+1} + h\left(Y_n \right)\ \ \ \text{ with } h(y)\, =\, y+ \log(1+\exp(-y))\, .
\end{equation}
Note in particular that, since $h(\cdot)$ is positive and increasing, we have  
\begin{equation}
\label{eq:Fn-onesided}
\begin{split}
F_{Y_{n+1}}(x)\, &=\, \bbP\left( \logZ_{n+1}+ h(Y_n) \le x, \, x- \logZ_{n+1} >0 \right)
\\&=\,  \bbP\left(Y_n \le h^{-1}\left(x-  \logZ_{n+1}\right), \,  \logZ_{n+1} <x \right)
\\
&=\, \int \ind_{(-\infty, x)} (z) F_{Y_n}\left( h^{-1}\left(
x-z\right)\right) \zeta(\dd z)\, .
\end{split}
\end{equation}
Although $h^{-1}(x) = \log(e^x-1)$, this precise form is not particularly important.

\medskip

\begin{proposition}
\label{th:Y}
Assume \eqref{eq:critical} and  that there exists $p>2$ such that $\logZ\in \bbL^p$. Then 
	the MC $Y$ has a unique nonzero invariant Radon measure $\nu$. It satisfies that ${\nu((-\infty,x])< \infty}$ for every $x\in \bbR$, 
	$\nu((-\infty,x])\le \nu((-\infty,0]) \bbP(\logZ< x)$ for $x\le 0$ and 
	$\nu(\bbR)= \infty$. Moreover, $Y$ is recurrent in the sense that if $O\subset \bbR$ is an open set with $\nu (O)>0$, then, given any $Y_0$,
	$\bbP(Y_n \in O \text{ i.o.})=1$.
\end{proposition}
\medskip

For nonnormalizable measures uniqueness is of course meant \emph{up to  a multiplicative constant} and we note that $\nu$ is characterized by 
\begin{equation}
\label{eq:fornu}
\int  g(z) \nu (\dd z)\, =\, \iint  g(z +h(y)) \gz(\dd z)  \, \nu (\dd y)\ \ \text{ for every measurable } g \ge 0\,  .
\end{equation}
We recall that a  measure $\nu$ on $\bbR$, equipped with the Borel $\gs$-algebra $\cB( \bbR)$,  is Radon if 
the measure of compact sets is finite and if it is regular, i.e.\ 
if for every $B \in \cB(\bbR)$ we have $\nu(B)$ is equal both to the infimum  of $\nu(O)$ taken over $O$ open that contains $B$,
and to the supremum  of $\nu(K)$ taken over $K$ compact contained in $B$.

\medskip

\begin{proof}[Proof of Prop.~\ref{th:Y}]
If we set $R_n:= \exp(Y_n)$ then, with $Z_{n+1}=\exp(\logZ_{n+1})$,  $R_{n+1}= Z_{n+1} (R_n+1)$ and $R$ is a MC with state space $(0, \infty)$. This is a special case of the process considered in \cite{cf:BBE97}. So by 
\cite[Proposition~1.1 and Theorem~3.3]{cf:BBE97} we know that $R$ has a unique nonzero invariant Radon measure $m$ and that this measure satisfies $m((0, \infty))=\infty$. 
Of course the MC $R$ with state space $(0, \infty)$ is equivalent to the MC $Y$ with state space $\bbR$. 
In particular, if we introduce the measure $\nu$
by setting $\nu((s, t])=m((\exp(s), \exp(t)])$ for every $s<t$, it is straightforward to see that $\nu$ is invariant for $Y$ and that $\nu$ inherits the properties of $m$. By the bound \cite[(A.2)]{cf:GG22} one infers that $\nu((-\infty, t])< \infty$ for every $t \in \bbR$. Recurrence of the process is proven in  \cite[Cor.~4.2$(i)$]{cf:BBE97}. Finally, the explicit bound on $F_\nu(x):=\nu((-\infty,x])$ holds because $F_\nu$ satisfies the analogue of \eqref{eq:Fn-onesided}, that is
\begin{equation}
\label{eq:Fn-onesided-infty}
F_\nu(x)\, =\,  \int_{(-\infty, x)} F_{\nu}\left( h^{-1}\left(
x-z\right)\right) \zeta(\dd z)\, .
\end{equation}
Using  that $F_\nu(\cdot)$ and  $h^{-1}(\cdot)$ are nondecreasing 
and that $h^{-1}(y) < y$, for every $y$, we see that for $z<x\le 0$ we have
 $F_{\nu}\left( h^{-1}\left(
x-z\right)\right) \le F_{\nu}\left( h^{-1}\left(
-z\right)\right)\le F_{\nu}(-z) \le F_{\nu}(0)$, hence $F_\nu(x)\le F_{\nu}(0) \bbP( \logZ <x)$.
\end{proof}

\section{The proof of the main result}
\label{sec:proofofmaintheorem}

We give here a series of three statements that we prove in the next sections. We then prove Theorem~\ref{th:main} assuming the validity of these statements. 

\subsection{On the invariant measure of the $Y$ process}  

Recall that $(\logZ_n)_{n=1, 2, \ldots}$ is an  IID sequence with  $\logZ_1 \law \logZ$. We introduce
$(S_n)$  the corresponding random walk starting at $0$: $S_n:= \sum_{j=1}^n \logZ_j$ for $n =1,2, \ldots$ and 
$S_0=0$.  Of course this random walk is null recurrent. 

\medskip
Recall \eqref{eq:iterY} for the definition of the Markov chain $(Y_n)$. 
\begin{theorem}  
	\label{thm:stationary_asymptotics}
	Assume \eqref{eq:critical} and  {\bf (H-1)} of Theorem~\ref{th:main} plus the following two conditions: 
	 \begin{description}
	\item[(T-1)] there exists $\xi > 3$ such that $\e[(\logZ^+)^{\xi}] < \infty$; 
	\item [(T-2)]  $\e[\logZ^2]\in (0, \infty)$. 
	\end{description}
	   Then any nonzero invariant measure $\nu$ of $Y$  satisfies that
\begin{equation} 
\label{eq:stationary_asymptotics}
\nu((-\infty, x])\stackrel{x\to \infty}=  c_\nu x+ d_\nu + O(x^{-(\xi-3)})\, .
	  \end{equation}     
	  for appropriately chosen constants
	  $c_\nu>0$ and $d_\nu\in \bbR$. 
	  
	  Moreover, if  {\bf (T-1)} is replaced by 
	  \smallskip
	  
	  \noindent  {\bf (T-1$^\prime$):} 
	 there exists $c>0$ such that
	  $\e[\exp(c\,\logZ^+)] < \infty$;
	  \smallskip
	  
	  \noindent then \eqref{eq:stationary_asymptotics} holds with error term $O\big(\exp\big(-\gd x \big)\big)$, for some positive constant $\gd$.
	  \end{theorem}
	 \medskip

{
\begin{rem}
\label{rem:exact}
An anonymous referee pointed out to us that $\nu$ is explicit if the density of $\logZ$ is $z \mapsto \exp(z)/(1+ \exp(z))^2=1/(2 \cosh(z/2))^2$, a case satisfying hypothesis  {\bf (T-1$^\prime$)}. In fact, one directly verifies using \eqref{eq:Fn-onesided-infty} that in this case $F_\nu(x)= \log(1+ \exp(x))$. Then one directly checks that $F_\nu(x)=x+ O(\exp(-x))$ for $x \to \infty$ so \eqref{eq:stationary_asymptotics} holds with $c_\nu=1$ and, above all, $d_\nu=0$. 
We are not aware of other exact solutions.  
\end{rem}
}


       
The proof of \Cref{thm:stationary_asymptotics} is in \cref{sec:stationary}. 
When applying 
\Cref{thm:stationary_asymptotics} we will set $c_\nu=1$, see \eqref{eq:asymptlar}.

\subsection{Contraction estimates}

Next is the following improvement of \cite[Prop.~2.2]{cf:GG22}, based on a  result and approach developed by Q.~Moulard \cite{cf:QM}.

\medskip 

\begin{theorem}
Choose $\gz$ satisfying \eqref{eq:critical} and
	such  that $\gs^2:=\bbE[\logZ^2] < \infty$. Then
	there exists $\gG _0$  such that for every $\gG  \ge \gG _0$ and every
	 $G=G_1-G_2$, with $G_1$ and $G_2$ risk functions of $\bbL^1$ random variables, we have  
	\begin{equation} 
	\label{eq:bound}
	\sum_{n=0}^\infty \left \Vert T_{\gG,0}^n G \right\Vert_1 \, \le \,  \frac{\cfromQM}{\gs^2} \, \gG^2  \left \Vert G \right\Vert_1\, ,
	\end{equation}  where $T_{\gG,0}$ is defined in \eqref{eq:TG0_def}.
	\label{thm:contractive}
\end{theorem}

\medskip

If $G$ is the risk function of a random variable we introduce
\begin{equation}
\label{eq:L_k}
L_\gG [G]\,:=\, \int_ \bbR \frac{1} {1+e^{\gG -x}} G(x) \dd x\, ,
\end{equation}
and one readily sees that, from \eqref{eq:Lyap-k}, 
the    Lyapunov exponent $\cL(\gG )$ is equal to $L_\gG [G_{\nu_\gG }]$, and that
\begin{equation}
\label{eq:Ly}
\left \vert L_\gG [G_1] - L_\gG [G_2 ]\right \vert \, \le \, \Vert G_1 -G_2 \Vert_1\, ,
\end{equation}
whenever $G_1$ and $G_2$ are the risk functions of two random variables.
We thus have the following important corollary to \cref{thm:contractive}:
\medskip

\begin{cor}
\label{th:bound-cor}
With  $\gG _0$ as in \cref{thm:contractive},
for $\gG  \ge \gG _0$ and for any probability $\gamma $  such that 
$\int _\bbR \vert x\vert \gamma (\dd x) < \infty$ ($\gamma$ may depend on $\gG$), 
we have
\begin{equation} 
\label{eq:bound-cor}
\left \vert \cL(\gG )- L_\gG [G_\gamma] \right \vert \, \le\, \left \Vert G_{\nu_\gG } - G _ \gamma \right \Vert _1\, \le  \,
 \frac{\cfromQM}{\gs^2} \, \gG^2
 \left \Vert T_\gG G_\gamma - G_\gamma \right\Vert_1\, .
\end{equation}
\end{cor}

\begin{rem} 
\label{rem:sym2}
The alternative expression for $\cL(\gG)$ in \eqref{eq:Lyap-k2}  suggests introducing 
\begin{equation}
\label{eq:L_k2}
L^{\star}_\gG [G]\,:=\,  \int_ \bbR \frac{1} {1+e^{\gG +x}} (1-G(x)) \dd x
\end{equation}
instead of $L_\Gamma$.
In general $L_\gG [G]\neq L^{\star}_\gG [G]$,  but 
$L_\gG [G_{\nu_\gG}]= L^{\star}_\gG [G_{\nu_\gG}]$. Moreover \cref{th:bound-cor} applies verbatim to $L^{\star}_\gG[\, \cdot\, ]$. 
\end{rem}

\medskip
\noindent 
\emph{Proof.} Choose $X_0$ of law $\gamma$ and $X_1$ as in  \eqref{eq:XMC}. By construction $G_{X_1}=T_\gG G_\gamma$
and, by the observation made right after \eqref{eq:XMC}, we have that $\vert X_1\vert \le \gG+ \vert \logZ_1\vert \in \bbL^2$.
Since  the Wasserstein-1 distance between the two $\bbL^1$ random variables $X_0$ and $X_1$ is equal to $\Vert T_\gG G_\gamma - G_\gamma \Vert_1$, we have that $\Vert T_\gG G_\gamma - G_\gamma \Vert_1 < \infty$.

Since the MC $(X_n)$ is positive recurrent, $\lim_{N\to\infty} T_\gG^N  G_\gamma (x)=
G_{\nu_\gG }(x)$ for every $x \in \bbR$ which is a continuity point of $G_{\nu_\gG }(\cdot)$, hence this convergence holds out of a set 
that is at most countable.
Therefore, by Fatou's Lemma and \Cref{thm:contractive}, we have that 
\begin{eqnarray}
\left \Vert G_{\nu_\gG } - G _ \gamma \right \Vert _1 \, &\le &
  \liminf_{N\to\infty} \left \Vert T_\gG^N G_\gamma - G_\gamma \right\Vert_1
\\
&  \le &  
 \sum_{n =0}^\infty  \left \Vert T_{\gG,0}^n \left( T_\gG G_\gamma - G_\gamma\right) \right\Vert_1 \, \\
 &\le & 
  \frac{\cfromQM}{\gs^2} \, \gG^2
   \left \Vert T_\gG G_\gamma - G_\gamma \right\Vert_1\, ,
\end{eqnarray}
and the proof is complete by \cref{eq:Ly}.
\qed

\subsection{Approximating the Furstenberg probability and the proof of \Cref{th:main}}
We are now going to patch together the  (infinite!) invariant measure of the $Y$ process, translated to the left of $\gG$,
and  the  (also infinite, of course) invariant measure of the $Y$ process driven by $-\logZ$ instead of  $\logZ$, reflected with respect to the origin
and translated to the right of $\gG$, see Fig.~\ref{fig:2}, to form a probability measure that we call $\gamma_\gG$.
We expect (and will show) that $\gamma_\gG$ is close to the invariant probability $\nu_\gG$ of the $X$ MC. 

\begin{figure}[h]
\centering
\includegraphics[width=13.5 cm]{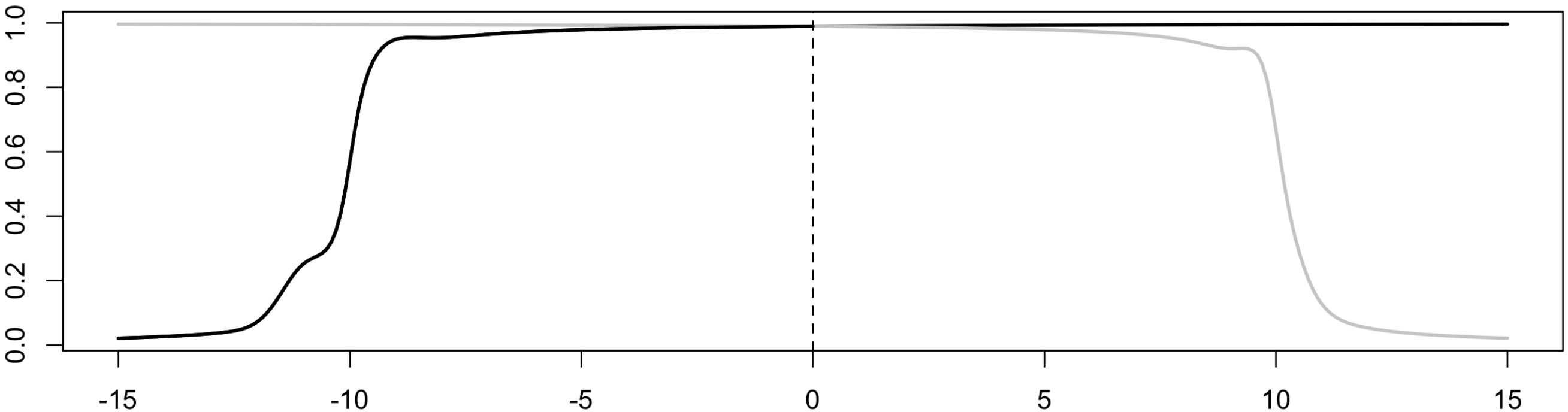}
\caption{\label{fig:2} For $\gG =10$ we have drawn in black (respectively, in grey) the density of the invariant measure of the $Y$ process
translated of $-\Gamma$ (respectively, of the  $Y$ process driven by $-\logZ$, reflected with respect to the origin and then  translated  of $+ \Gamma$): one can readily see that the densities of these two measure exist if $\logZ$ has a density. Note also that, in particular,the cumulative function associated to the black density function is $F_\lar(\cdot+\gG)$, 
Therefore, up to a multiplicative factor, the density of the probability $\gamma_\gG$ is given by the black line on the left of the origin and by the grey one on the right.
}
\end{figure}


Let $F_\lar(x) := c_\nu^{-1} \int \ind_{(-\infty,x]}(y) \nu (\dd y)$ for $x \in \bbR$, where $\nu$ and $c_\nu$ are as in \cref{thm:stationary_asymptotics}. So $F_\lar(\cdot)$ is the cumulative function of the (suitably normalized) stationary measure of $Y$  
and Theorem~\ref{thm:stationary_asymptotics} yields
\begin{equation}
	\label{eq:asymptlar}
	F_\lar(x)\stackrel{ x \to \infty}= x+ c_\lar + \begin{cases} O\left(x^{-(\xi-3)}\right) & \text{ under {\bf (T-1)}}, 
	\\
	O\left(\exp(-\gd x)\right) & \text{ under {\bf (T-1$^\prime$)}}.
	\end{cases}  
\end{equation}
Let $F_\rar$ be the same with $\logZ$ replaced by $-\logZ$, which describes the related object for the right edge instead of the left edge, which then has comparable asymptotic behavior.  Using this, we define an approximate stationary distribution $\gamma_\Gamma$ by 
\begin{equation}
\label{eq:DHprobability}
G_{\gamma_\gG} (x)\, :=\, \begin{cases} 
F_{\rar}(\gG -x) / C_\gG   & \text{ if } x \ge 0\, ,
\\
1- \left(F_\lar(x+\gG )/C_\gG \right) & \text{ if } x \le 0\, ,
\end{cases}\ \ \ \text{ with } C_\Gamma :=
	F_{\lar}(\gG) + F_{\rar}(\gG)\,.
\end{equation}
Note that $C_\Gamma$ has been selected so that $G_{\gamma_\gG}(0)$ is coherently chosen.
We register also that for the cumulative probability  $F_{\gamma_\gG} (\cdot)=1-G_{\gamma_\gG} (\cdot)$ we have
\begin{equation}
\label{eq:cumFk}
F_{\gamma_\gG} (x)\, :=\, \begin{cases} 
1-\left(F_{\rar}(\gG -x) / C_\gG \right)  & \text{ if } x \ge 0\, ,
\\
F_\lar(x+\gG )/C_\gG  & \text{ if } x \le 0\, .
\end{cases}
\end{equation}
Note also that, by the explicit bound on $\nu((-\infty,x])$ for $x \le 0$ in  Prop.~\ref{th:Y}, the tail behavior(s) of $\gamma_\gG$
is controlled by the tail behavior(s) of $\logZ$, in particular we have $\int_\bbR \vert x \vert \gamma_\gG(\dd x) < \infty$, which is required to apply 
\cref{th:bound-cor}.

Set
\begin{equation}
\label{eq:kappa2}
\kappa_2\, :=\, \frac{c_\lar +c_\rar}2\, .
\end{equation}
{It is worth pointing out that in the exactly computable case of Remark~\ref{rem:exact} we have $c_\lar$ is equal to $c_\rar$, because the law of $\logZ$ is symmetric,
and they are equal to zero. Hence $\kappa_2=0$. While we do not know other exactly computable cases, in \cite[(8.1.11), see also (4.1.9)]{cf:Othesis}
one can find a conjecture for the exact value of $\kappa_2$ that involves the ladder and excursion process of the random walk with increments $\logZ$. Moreover in 
\cite[(4.34)]{cf:NL86} one finds an explicit expression  for $\kappa_2$ for a specific family of laws of $\logZ$: also in this second case the expression is a conjecture from a mathematics standpoint. All these expressions point to the dependence of $\kappa_2$ on fine details of the law of $\logZ$.}

\medskip

\begin{proposition}
\label{thm:onestep}
	Assume that both $\logZ$ and $-\logZ$ satisfy the assumptions  of \cref{thm:stationary_asymptotics} (in particular, there exists $\xi>3$ such that $\logZ\in \bbL^{\xi}$, so {\bf (T-2)} becomes superfluous).
	Then for  a suitably chosen $\gd>0$
\begin{equation}
\label{eq:Ck}
	C_\Gamma 
	\stackrel{\Gamma \to \infty}=
	2\gG  + 2 \kappa_2 + \begin{cases} O\left(\gG^{-(\xi -3)}\right) & \text{ under {\bf (T-1)} for both $\pm \logZ$}, 
	\\
	O\left(\exp(-\gd \gG)\right) & \text{ under {\bf (T-1$^\prime$)} for both $\pm \logZ$ }.
	\end{cases} 
\end{equation}  
Moreover
	\begin{equation}
		\label{eq:Lyap}
		L_\gG [G_{\gamma_\gG} ]\,=\, \frac 1{C_\gG }\int_\bbR \frac{F_\rar(y)}{1+e^y} \dd y+ O(\exp(-\gG ))
		\,,
	\end{equation}
		and
	\begin{equation}
		\left\| T_\gG  G_{\gamma_\gG}  - G_{\gamma_\gG}  \right\|_1
		\,=\,
		\begin{cases} O\left(\gG^{-\xi+2}\right) & \text{ under {\bf (T-1)} for both $\pm \logZ$}, 
	\\
	O\left(\exp(-\gd \gG)\right) & \text{ under {\bf (T-1$^\prime$)} for both $\pm \logZ$}.
	\end{cases} 
		\label{eq:onestep}
	\end{equation}
\end{proposition}
\medskip

The proof of \Cref{thm:onestep} is in Section~\ref{sec:onestep}.
We have now all the ingredients for the proof of \Cref{th:main}.

\medskip

\begin{proof}
\Cref{th:main} follows immediately by combining \cref{th:bound-cor}, with the choice of $\gamma=\gamma_\gG$, and \cref{thm:onestep}: note that the value of $\gd$ in \Cref{th:main} may be chosen equal to $\gd/2$, with $\gd$ the constant appearing in \cref{thm:onestep}.
 In particular by \eqref{eq:Lyap} we obtain
\begin{equation}
	\kappa_1
	=
	\frac12
	\int_\bbR \frac{F_\rar(y)}{1+e^{y}} \dd y\, . \label{eq:kappa_1_formula-0}
\end{equation}
\noindent We note that $\kappa_1$ can also be evaluated using $F_\lar$ instead of $F_\rar$, as shown in the forthcoming Remark \ref{rem:sym3}.
\end{proof}

\section{The stationary measure of the $Y$ process}
\label{sec:stationary}

 This section is devoted to 
 the proof of Theorem~\ref{thm:stationary_asymptotics}. Let $\nu$ be a nonzero invariant measure of $Y$. By \cref{th:Y}, we have $\nu((-\infty, x])< \infty$ for every $x\in\R$. 
 
 We state a first lemma which rephrases the result of Theorem~\ref{thm:stationary_asymptotics}.
 
 \begin{lemma} Let  $\xi>3$ and $c_\nu>0$. Then, the existence of $d_\nu \in \R$ such that equation \eqref{eq:stationary_asymptotics} holds is equivalent to the fact that, as $x\to \infty$, uniformly in $x\le y \le 2x$ we have
\begin{equation} 
\label{eq:stationary_asymptotics_equiv}
 \nu((x, y])=  c_\nu (y-x)+ O\big(x^{-(\xi-3)}\big)\, .
\end{equation}
Similarly, let $c_\nu>0$, then the existence of $d_\nu \in \R$ and $\delta>0$ such that equation \eqref{eq:stationary_asymptotics} with error term replaced by $O\big(e^{-\gd x}\big)$ holds, is equivalent to the existence of $\delta>0$, such that, as $x\to \infty$, uniformly in $x\le y \le 2x$:
\begin{equation} 
\label{eq:stationary_asymptotics_equiv_expo}
 \nu((x, y])=  c_\nu (y-x)+ O\left(\exp\big(-\gd x \big)\right).
\end{equation}
\end{lemma}

\begin{proof}
We prove the first statement. 
Writing $\nu((x, y])=\nu((-\infty, y])-\nu((-\infty, x])$, it is clear that \eqref{eq:stationary_asymptotics} implies \eqref{eq:stationary_asymptotics_equiv}. Conversely, assume that \eqref{eq:stationary_asymptotics_equiv} holds. For any $x\ge 1$, consider the (unique) integer $n$ such that $2^n < x \le 2^{n+1}$ and write
\begin{equation}
\begin{aligned}
\nu((-\infty, x])- c_\nu x = & \;  \nu\big((-\infty, 1]\big)-c_\nu\\
& + \sum_{i=0}^{n-1} 
\nu\big(\big(2^i, 2^{i+1}\big]\big)- c_\nu \big(2^{i+1}-2^i\big)\\
& +\nu\big(\big(2^n , x\big]\big)- c_\nu \big(x-2^n\big)\, ,
\end{aligned}
\end{equation}
and note that, due to \eqref{eq:stationary_asymptotics_equiv}, the series
$\sum_{i\ge 0} 
\left(\nu([2^i, 2^{i+1}])- c_\nu (2^{i+1}-2^i) \right)$ 
converges and its partial remainder 
$\sum_{i\ge n}
\left(\nu([2^i, 2^{i+1}])- c_\nu (2^{i+1}-2^i) \right)=O(2^{-n(\xi-3)})$. By \eqref{eq:stationary_asymptotics_equiv}, the term $\nu((2^n , x])- c_\nu (x-2^n)$ is also a $O(2^{-n(\xi-3)})$. Using finally that $O(2^{-n(\xi-3)})=O(x^{-(\xi-3)})$,
we get that \eqref{eq:stationary_asymptotics} holds with $d_\nu:= \nu((-\infty, 1])-c_\nu + \sum_{i=0}^{\infty} 
\nu((2^i, 2^{i+1}])- c_\nu (2^{i+1}-2^i) $.

The proof is similar in the case of error term $O\big(e^{-\delta x}\big)$.
\end{proof}

\bigskip

Our focus will now be proving \eqref{eq:stationary_asymptotics_equiv}. There will be four preliminary steps:
 \medskip
 \begin{enumerate}[leftmargin=0.7 cm]
 \item 
 By elaborating on the well-known representation of  the invariant measure of a recurrent MC (see for example \cite[Sec.~6 of Ch.~3]{cf:MC}), we show that $\nu((x,y])$ can be represented  by using the {{\sl excursion occupation times}} $V_{a}(b, (x, y])$ of $Y$, see the  equation \eqref{pi-V} below. In  Lemma \ref{l:Theta}, we  give   an expression of $V_{a}(b, (x, y])$ in terms of the strictly descending epochs of the random walk $S$ and a crucial  function $F$, that  links $V_{a}(b, (x, y])$ and the $Y$ process observed at the \emph{ladder epochs} of $S$ (see \eqref{eq:ladder-times}).
 \item By considering the weak ascending ladder epochs of $S$, we rewrite $F$ in terms of a new process $(J_k(\theta))_{k\ge 0}$, defined in \eqref{eq:defJ},  which is now an increasing Markov chain.
 \item
  We determine the asymptotic behavior of  $F$, see Proposition \ref{P:F}: this is the main technical estimate in the proof.
 \item We perform some technical estimates that are central in controlling the error 
  arising when replacing $F$ by its asymptotic behavior into
   the representation formula for the cumulative function of $\nu$.
 \end{enumerate} 
 
 \medskip
 
 We recall that we work under the assumptions of \Cref{thm:stationary_asymptotics}, that is
 the criticality condition \eqref{eq:critical}, {\bf (H-1)} of \Cref{th:main}, 
   {\bf (T-1)} or  {\bf (T-1$^\prime$)}, and  {\bf (T-2)} explicitly stated in \Cref{thm:stationary_asymptotics}.

 \medskip
 
 \subsection{Step 1: representation of the invariant measure} 
  First,  we observe that $Y$ can be explicitly solved. By induction we see that, when $Y_0=\theta$,    for  $n=0, 1,2,\ldots$
  \begin{equation}
	Y_n\,=\,  \log \left( e^{\theta+ S_n}+ \sum_{i=0}^{n-1} e^{S_n-S_i }\right) =  \theta+ S_n +  \log \left( 1+ e^{-\theta} \sum_{i=0}^{n-1} e^{-S_i }\right) \, ,
	\label{eq:Yn_direct}
\end{equation}
where as before $S_n = \sum_{i=1}^n \mathtt{z}_i$ and we adopt  the convention $\sum_\emptyset:=0$.

Since $\nu$ is nonzero, let us fix an arbitrary $a >0$ such that  $\nu((-\infty, a]) >0$
and denote by $\tau_Y(a):=\inf\{n\ge 1: Y_n \in (-\infty, a]\}$ the first return time to $(-\infty,  a]$ of $Y$. 
By applying \cite[Theorem~3.6.5]{cf:MC} to the recurrent Markov chain  $Y$,  and by applying the Markov property at time 1,
we have  that for any Borel set $I\subset (a, \infty)$,
 \begin{align}
    \nu(I) & = \int_{ (-\infty,  a]} \e_{v} \left[\sum_{n=0}^{\tau_Y(a)-1} \ind_{\{Y_n \in I\}}\right] \nu(dv) \nonumber \\ 
& = \int_{(-\infty, a]} \e_{v}\big[ V_{a}(Y_1, I) \ind_{\{Y_1> a\}} \big]\nu(\dd v)  , 
	\label{pi-V}
 \end{align}
\noindent where, for $b>a$ and any Borel set $I\subset (a, \infty]$, we set
\begin{equation}  
	V_{a}(b, I)\,:=\, \bbE_b\left[\sum_{n=0}^{ \tau_Y(a)-1}\ind_{  \left\{ Y_n \in I\right\}}\right]\, \,=\,  \sum_{n=0}^\infty \p_b\Big( \tau_Y(a)> n, Y_n \in I\Big).
\label{def-V}
\end{equation}
 Later on, we will take $I$ of the form $(x, y]$, with $a<x\le y \le 2x$, but let us remain general for now.
 
 \medskip 
 
We are now going to give a representation of $V_a(b, I)$ using the fluctuation theory for the one-dimensional random walk $S$. Consider the strictly descending ladder times $\varrho_0:=0$ and, for $j\ge 1$
\begin{equation}
\label{eq:ladder-times}
\varrho_j\,:=\,\min\left\{n\ge \varrho_{j-1}: S_n < S_{\varrho_{j-1}}\right\}\,.
\end{equation}  
Furthermore, consider the process $\Theta=(\Theta_j)_{j\ge 0}$ defined by $\Theta_j:= Y_{\varrho_j}$. Let
\begin{equation}   
	\eta_j\,:=\, \sum_{n=\varrho_{j-1} }^{\varrho_j-1} e^{  S_{\varrho_{j}}-S_n }
	 \qquad  \text{ and }  \qquad 
	\Delta_j\,=\,  S_{\varrho_{j-1}}-S_{\varrho_j} >0\,. 
	\label{def-eta-Delta}
\end{equation}
 By the strong Markov property of $S$,  $((\eta_k, \Delta_k))_{k\ge1}$ is an IID sequence adapted to the filtration $({\mathcal F}_{\varrho_j})_{j\ge 0}$, where ${\mathcal F}_n:= \sigma(\{S_i, i\le n\})$ for $n\ge 0$. Using \eqref{eq:Yn_direct}, we have that  
 \begin{equation}
  \Theta_{j+1}= \log\Big(   e^{- \Delta_{j+1}}  e^{\Theta_j}  +  \eta_{j+1}  \Big) , \label{Thetaj+1}  \end{equation}
 which, in view of the fact that $(\eta_j, \Delta_j)_{j\ge 1}$ is an IID sequence, implies that $(\Theta_j)_{j\ge 0}$ is a Markov chain. For completeness, let us mention that, if $\Theta_0= Y_0=b$, then $\Theta_j$ is given by:
 \begin{equation} 
 \label{eq:defTheta}
 \Theta_j= Y_{\varrho_j}
 =
 \log \left(\, e^{b+S_{\varrho_j}} + \sum_{k=1}^j e^{S_{\varrho_j}- S_{\varrho_k}} \eta_k\right)\, .
 \end{equation}

\medskip
  
\begin{lemma}  \label{l:Theta} 
 For  $b>a$, for any Borel set $I\subset (a, \infty)$, 
   \begin{equation}  V_{a}(b, I)= \e_b \left[  \sum_{j=0}^{\tau_\Theta(a)-1} F(\Theta_j, I) \right]= \sum_{j=0}^\infty \e_b\left[\ind_{\{\tau_{\Theta}(a)>j\}} F(\Theta_j , I)\right],  \label{V(x)}\end{equation}
  where $\tau_\Theta(a)$ is the first return time to $(-\infty, a]$ of $\Theta$: 
  \begin{equation}   
	  \label{eq:def-tau_Theta}
	  \tau_\Theta(a) :=\inf\{j\ge 1: \Theta_j \le a\}, 
  \end{equation}
  and $F(\theta, I)$ is defined, for any $\theta\in\R$ and any Borel set $I\subset \R$, by:
  \begin{equation}
  F(\theta, I):=  \e_\theta \left[ \sum_{n=0}^{\varrho_1-1} \ind_{\{ Y_n\in I\}}\right]= \sum_{n=0}^\infty \p_\theta(\rho_1>n, Y_n\in I)  .  \label{def-F} 
  \end{equation}
 \end{lemma}
 \medskip


 \medskip
  \begin{proof}
 We use the definition of $Y$ to see that $Y_n \ge Y_{\varrho_j} + S_n-S_{\rho_j} \ge  Y_{\varrho_j} =\Theta_j$  for $\varrho_j \le n < \varrho_{j+1}$. This means that in fact  $\tau_Y(a)$ must be equal to $\varrho_{\tau_\Theta(a)}$, with $\tau_\Theta(a)$ defined in \eqref{eq:def-tau_Theta}. Hence
\begin{equation}
	\begin{split}
		V_{a, b}(I)&=     \e_b \left[ \sum_{j=0}^{\tau_\Theta(a)-1} \sum_{n=\varrho_{j}}^{\varrho_{j+1}-1} \ind_{\{Y_n \in I\}}\right]
 \\
 &=  \sum_{j=0}^\infty \e_b \left[ \ind_{\{\tau_\Theta(a)>j \}} \sum_{n=\varrho_{j}}^{\varrho_{j+1}-1} \ind_{\{Y_n \in I\}}\right].
	\end{split}
\end{equation}
   Observe that the event $\{\tau_\Theta(a)> j\}$ is measurable with respect to ${\mathcal F}_{\varrho_{j}}$,
and that, due to the Markov property of $Y$, the conditional expectation of $ \sum_{n=\varrho_{j}}^{\varrho_{j+1}-1} \ind_{\{Y_n \in I\}}$ with respect to ${\mathcal F}_{\varrho_{j}}$,
 is equal to $F(Y_{\varrho_{j}}, I)$. We obtain  \eqref{V(x)} and  the proof of Lemma \ref{l:Theta} is complete.
 \end{proof}
  
\subsection{Step 2: representation of the function $F$ in terms of an auxiliary process}
Recall \eqref{def-F} for the definition of $F(\theta, I)$. We are going to use again the fluctuation theory for the random walk $S$ to give a representation of $F$ in terms of an \emph{increasing} process. Let us introduce some notations.

  Let  $H_k:= S_{\alpha_k}$ and $\alpha_k, k\ge0,$ be the weak ascending ladder heights and epochs of $S$: $\alpha_0:=0$ and for any $k\ge 1$ 
  \begin{equation} \label{eq:ladder-alpha}
  \alpha_k:= \inf\left\{n>\alpha_{k-1}:\, S_n \ge S_{\alpha_{k-1}}\right\}\,.
  \end{equation} 

Let us also introduce for $ k \ge 1$:
\begin{equation}
 \widetilde{\Delta}_k:= H_{k}-H_{k-1} \ge 0 \,, \qquad \widetilde\eta_k  \,:=\, \sum_{n= \alpha_{k-1}+1}^{\alpha_k} e^{S_n-H_k}\, .\label{eq:defetatilde}
 \end{equation}
By the Markov property of $S$, the sequence $((\widetilde{\Delta}_k, \widetilde\eta_k))_{k\ge 1}$ is IID. Let us define a new Markov chain $(J_k(\theta))_k$, adapted to the filtration $(\mathcal{F}_{\alpha_k})_{k\ge 0}$,   by $J_0(\theta):=\theta$, and for all $k=0, 1, 2, \dots$
\begin{equation}
J_{k+1}(\theta):=J_k(\theta) + \widetilde{\Delta}_{k+1} + \log(1 + \widetilde\eta_{k+1} e^{-J_k(\theta)} ).
\end{equation}
Note that $J_k(\theta)$ is increasing in $k$. Explicitly, $J_k(\theta)$ is given by:
\begin{equation}
J_k(\theta) \, =\, \log \Big(e^{ \theta+  H_k} + \sum_{j=1}^k \widetilde\eta_j  \, e^{H_k-   H_{j-1}}\Big) =   \theta+  H_k+ \log \Big(1 + e^{-\theta} \sum_{j=1}^k \widetilde\eta_j  \, e^{-   H_{j-1}}\Big). \label{eq:defJ}
\end{equation}

\begin{lemma}\label{l:rewriteF} 
For any $\theta\in\R$, any Borel set $ I\subset \R$, we have \begin{equation} 
   F(\theta,I) =  \sum_{k=0}^\infty \p(J_k(\theta) \in I)\, .
  \label{eq:rewriteF}
  \end{equation}
\end{lemma}

\medskip

  \begin{proof}  The proof is based on the duality lemma for the random walk $S$. By the definitions of $ F(\theta, I)$ and $\varrho_1$, we have   
  \begin{equation*}
	  \begin{split}
		  F(\theta, I)&= \sum_{n=0}^\infty  \p \Bigg(S_1\ge 0, ..., S_n \ge 0,  \log \Big(e^{\theta+S_n} +  \sum_{j=0}^{n-1} e^{ S_n- S_j}\Big) \in I\Bigg)
 \\
 &= \sum_{n=0}^\infty  \p \Bigg(S_1\ge 0, ..., S_n \ge 0,  \log \Big(e^{\theta+S_n} +  \sum_{j=1}^{n} e^{ S_n- S_{n-j}}\Big) \in I\Bigg).
	  \end{split}
  \end{equation*}
  
\noindent Since $(S_n-S_{n-j})_{0\le j\le n}$ has the same distribution as $(S_j)_{0\le j\le n}$, we obtain that 
   \begin{equation*}
	  \begin{split}
		  F(\theta, I)  &=\sum_{n=0}^\infty  \p \Bigg( S_n\ge S_{n-1}, ..., S_n \ge S_1, S_n\ge 0,  \,   \log \Big(e^{\theta+  S_n} +  \sum_{j=1}^n e^{S_j}\Big) \in I\Bigg)\\
 &= \sum_{k=0}^\infty \p\Bigg(\log \Big(e^{\theta+  H_k} +  \sum_{j=1}^{\alpha_k} e^{S_j}\Big) \in I\Bigg)\\
  &= \sum_{k=0}^\infty \p\Bigg(\theta+  H_k + \log\Big(1 + e^{-\theta} \sum_{j=1}^{\alpha_k} e^{S_j-H_k}\Big) \in I\Bigg).
	  \end{split}
  \end{equation*}

\noindent Note that \begin{equation}\sum_{j=1}^{\alpha_k} e^{  S_j-  H_k  }= \sum_{\ell=1}^k \sum_{j=\alpha_{\ell-1}+1}^{\alpha_\ell} e^{  S_j-  H_\ell  } e^{ H_\ell -H_k }=   \sum_{\ell=1}^k \widetilde\eta_\ell \, e^{ H_\ell-H_k }, \end{equation}

\noindent by definition of $\widetilde \eta$ in \eqref{eq:defetatilde}. The strong Markov property of $S$ yields that $(\widetilde\eta_\ell, H_\ell-H_{\ell-1})_{\ell\ge1}$ are IID. This yields that $(H_k, \sum_{\ell=1}^k \widetilde\eta_\ell \, e^{ H_\ell-H_k })$ is distributed as $(H_k, \sum_{j=1}^k \widetilde\eta_j  \, e^{-   H_{j-1}})$ and therefore that $\theta+  H_k + \log(1 + e^{-\theta} \sum_{j=1}^{\alpha_k} e^{S_j-H_k}) $ is distributed as $J_k(\theta)$. This concludes the proof.
 \end{proof}

\subsection{Step 3: the asymptotic behavior of the function $F$} 
 
We now take $I=(x, y]$, with $0\le x \le y \le 2x$. We are going to determine the asymptotic behavior of $F(\theta, (x, y])$ as $x\to\infty$ by using the fact that $J_k(\theta)$ behaves asymptotically like a random walk with increments having the law of $H_1$, which is a nonnegative random variable. 
The renewal function of $H$ is defined as  
\begin{equation} 
R(x):=  \sum_{k=0}^\infty \p\Big( H_k \le x\Big) , \qquad x\ge 0, \label{def-R}
\end{equation} 
and $R(x):=0$ for $x<0$.

\medskip
The following Lemma, due to Rogozin \cite{cf:Rogozin} and Stone \cite{cf:Stone}, plays a crucial role in the study of $F$. 
\medskip

\begin{lemma}[\cite{cf:Rogozin}, formula (28), and  \cite{cf:Stone}]
	\label{lem:Rogozin}
	Assume that
	
\noindent{\bf (R-1):} there exists  $n_0\ge1$ such that the law of $H_{n_0}$ has an absolutely continuous part: this property is normally referred to as ``the law of $H_1$ is spread out";
 
\noindent and either

\noindent{\bf (R-2):} there exists  $\kappa\ge 2$ such that $\e[H_1^\kappa] < \infty$;
 
\noindent or  

\noindent{\bf (R-2$^\prime$):} there exists  $c>0$ such that $\e[e^{c H_1}] < \infty$.

Then for any $x>0$,  \begin{equation}\label{Rphi} \left|R(x) -  \left(c_R\,  x  + c'_R\right)\right|\, \le \, \varphi(x):=  
\begin{cases}
C_\varphi \, (1+x)^{-(\kappa-2)}, \qquad &\mbox{under {\bf (R-2)}} , \\
C_\varphi\, e^{-C'_\varphi \,x}, \qquad &\mbox{under {\bf (R-2$^\prime$)}} , 
\end{cases}   
 \end{equation}
 
 \noindent where   $C_\varphi, C'_\varphi$ are two (unimportant) positive constants and
\begin{equation}\label{eq:cRc'R}
c_R\,:=\,\frac1{ \e(H_1)}, \qquad c'_R\,:=\, \frac{\e\left[H_1^2\right]}{2\e\left[H_1\right]^2}.\end{equation}
    \end{lemma}
    
    \medskip
    
\begin{rem}\label{rem:rogozin}

  (i)  By the main theorem in Doney \cite{cf:Doney80}, 
  
  $\bullet$    {\bf (R-2)}  is satisfied if $\e[(\logZ^+)^{\kappa+1}]<\infty$, hence it is satisfied under hypothesis {\bf (T-1)} of \cref{thm:stationary_asymptotics}   with $\kappa=\xi-1\ge 2$;
  
  $\bullet$    {\bf (R-2$^\prime$)}  is satisfied under {\bf (T-1$^\prime$)}.

  (ii) The case  {\bf (R-2$^\prime$)} of Lemma \ref{lem:Rogozin} is stated in Stone \cite{cf:Stone} with an implicit constant that must be equal to  $c'_R$. See also \cite[Sec.~VII.2 and Ex.~VII.2.2]{cf:Asm}.
  
  The case   {\bf (R-2)}  is also obtained in Carlsson \cite[Corollary 2 and Remark 2]{cf:Carlsson} under the slightly weaker assumption that $H_1$ is strongly nonlattice, instead of {\bf (R-1)}. Here, we follow the formalism of Rogozin \cite{cf:Rogozin}, which states that under  {\bf (R-2)}, \begin{equation} 
  R(x)\stackrel{x\to \infty}= c_R\, x  + c'_R - \frac1{2(\e(H_1))^2} \e \left[\left(H_1-x\right)^2\ind_{\{H_1\ge x\}}\right]+ o\big(x^{-\kappa+1}\big).  \label{R(x)}\end{equation}
 Notice that the condition {\bf (R-1)} is equivalent to  \cite[Condition (6)]{cf:Rogozin}, by  Remark 4 there. Using {\bf (R-2)} we see that $\e[H_1^2 \ind_{\{H_1\ge x\}}] \le \e(H_1^\kappa) x^{2-\kappa}$, hence there exists $C_R>0$ such that 
 \begin{equation} \left|R(x) -  \left(c_R\,  x  + c'_R\right)\right|\, \le \, C_R  x^{2-\kappa}\ \text{ and } \ R(x)\, \le \, C_R\, x \qquad \text{ for } x\ge 1. 
 \label{R(x)2}\end{equation}
 Of course \eqref{R(x)2} is also valid in the case  {\bf (R-2$^\prime$)}, and $C_R$ is just a constant involved in  upper bounds on error terms, while $c_R$ and $c'_R$ are the constants in \eqref{eq:cRc'R}.
  \end{rem}

 The following Proposition gives the key estimate  in the proof of Theorem \ref{thm:stationary_asymptotics}.     Recall \eqref{def-F} for the definition of $F$. 
      
      \medskip

  \begin{proposition}\label{P:F}  Assume {\bf (R-1)} and {\bf (R-2)} with $\kappa\ge 2$. Then, for $x\to \infty$ and uniformly in $0\le \theta\le x/2$ and $x\le y \le 2x$,  we have 
 \begin{equation}  \label{F-asymp}
 F(\theta, (x, y]) \,=\,   c_R\, (y-x) + O\big(  x^{-(\kappa-2)}\big)\, .
\end{equation}
Assuming  {\bf (R-2$^\prime$)} instead of {\bf (R-2)}, there exists a positive constant $\delta$ such that \eqref{F-asymp} holds with error term replaced by $O\left(     e^{-\delta \, x  }\right)$.
     \end{proposition}
  
     \medskip 
     
     

\medskip

Before presenting the proof of Proposition \ref{P:F}, we state two lemmas. The first lemma will allow us to jump from initial point $\theta$ to some larger (random) initial point. From this larger initial point, it will become possible to compare $F$ with $R$: upper and lower bounds are given in the second lemma, which is useful really for large $\theta$.

\begin{lemma}\label{lem:MarkovJ}
For every $\theta\in \R$, any Borel set $I\subset \R$ and any $m\ge 1$, we have
 \begin{align}
 F(\theta, I) & =   F_m(\theta, I) +  \e \left[F(J_m(\theta), I) \right]
 \end{align}
where \begin{align}
 F_m(\theta, I) & \, := \, \sum_{k=0}^{m-1} \p(J_k(\theta) \in I)\, .
\end{align}
\end{lemma}

\begin{proof}
This is just an application of the Markov property at time $m$ for the Markov chain $(J_k(\theta))_{k\ge 0}$.
\end{proof}

\begin{lemma}\label{lem:boundFR}
For every $x$ and $y$ such that $0\le x\le y \le 2x$, for every $\theta\ge 0$ and $t>0$, we have
 \begin{equation}
 \begin{split}
 F(\theta, (x, y]) & \le R\left(y-\theta\right) -R\left(x- \theta - t\right)+ \Lambda(2x, e^{\theta} t) \\
 F(\theta, (x, y])  & \ge  R\left(y- \theta-t \right)-  R\left(x-\theta \right) - \Lambda(2x, e^{\theta} t)\, ,
 \end{split}
 \end{equation}
where  
\begin{align}\label{eq:defGamma}
\Lambda(x, t) & \, := \, \sum_{k=1}^\infty \p\Big( H_k\le x, \sum_{j=1}^\infty \widetilde\eta_j  \, e^{-   H_{j-1}}  > t \Big).
\end{align}
\end{lemma}

\begin{proof}
We use \eqref{eq:rewriteF} for an expression of $F(\theta, (x, y])$. By \eqref{eq:defJ}, we have 
\begin{equation}H_k+\theta\le J_k(\theta)\le H_k+\theta+e^{-\theta}  \sum_{j=1}^\infty \widetilde\eta_j  \, e^{-   H_{j-1}}.\end{equation}
Hence for $z\in \{x, y\}$, 
\begin{equation}\p(J_k(\theta)\le z) \le \p(H_k\le z-\theta)\end{equation}
and
\begin{equation}
\begin{split}
\p\big(J_k(\theta)\le z\big) 
 & \ge \p\Big(H_k+\theta +e^{-\theta}  \sum_{j=1}^\infty \widetilde\eta_j  \, e^{-   H_{j-1}} \le z\Big)\\
& \ge  \p\Big(H_k+\theta + t\le z, \sum_{j=1}^\infty \widetilde\eta_j  \, e^{-   H_{j-1}} \le e^{\theta}t\Big)\\
& \ge  \p\big(H_k+\theta + t\le z\big)-\p\Big(H_k \le 2 x, \sum_{j=1}^\infty \widetilde\eta_j  \, e^{-   H_{j-1}} > e^{\theta}t \Big)\,,
\end{split}
\end{equation}
using $z-\theta-t\le 2x$ in the last line. Applying these two inequalities with $z=y$ then $z=x$, and summing over $k$,  we get the lower and the upper bounds in the lemma.
  \end{proof}

\medskip

{\noindent\it Proof of Proposition \ref{P:F}.}  The asymptotic behavior  of $F$ will follow from Lemmas \ref{lem:MarkovJ} and \ref{lem:boundFR} once we have obtained the estimates on $F_m(\theta,(x,y]), \e [ R(y- J_m(\theta))-R(x-J_m(\theta)-t)], \e [ R(y- J_m(\theta)-t)-R(x-J_m(\theta))]$ and $\e[\Lambda(2x, e^{J_m(\theta)}t)]$. This will be done separately in what follows. 

For ease of exposition we introduce the symbol   $O_{\mathtt{u}}(f(s))$ to denote a quantity whose absolute value is bounded by   $C |f(s)|$ with  $C>0$ independent of the parameter $s$. In general $s$ may be of dimension larger than one, and belong to a domain that will be specified. 

\medskip
\subsubsection{Estimate of $F_m(\theta,(x, y])$} Since $F_m(\theta,(x, y])\, \le \,
\sum_{k=0}^{m-1} \p\left(J_k(\theta) >x\right) \le m  \p\left(J_m(\theta)>x\right)$, our goal is to estimate $\p\left(J_m(\theta) > x\right)$. 
We observe that for all $x\ge 4$, $\theta \in [0, x/2]$ and $m\ge 1$, if $H_m \le \frac{x}{4}$ and $\sum_{j=1}^\infty \widetilde\eta_j  \, e^{-  H_{j-1}} \le  e^{x/2}$, then
$$J_m(\theta)\,\le\,     H_m+\log\Bigg(e^{\theta} + \sum_{j=1}^\infty \widetilde\eta_j  \, e^{-   H_{j-1}}\Bigg)\le \frac{x}4 + \log\big(2e^{\frac{x}2}\big) \le x.$$
Hence, for all  for all $x\ge 4$, $\theta \in [0, x/2]$ and $m\ge 1$,
 \begin{align}
\p(J_m(\theta) > x) \, &\le\, \p\left( H_m > \frac{x}{4}\right) + \p\left(\sum_{j=1}^\infty \widetilde\eta_j  \, e^{-  H_{j-1}}> e^{x/2}\right) \label{YkIJ} 
\end{align}

We first estimate the second term in \eqref{YkIJ}, as the estimation of first term will depend on whether the case 
{\bf (R-2)} or {\bf (R-2$^\prime$)} applies.

Remark that $\widetilde \eta_1= 1+ \sum_{i=1}^{\alpha_1-1} e^{S_i - H_1} \le 1 + \sum_{i=1}^\infty e^{S_i} {\bf 1}_{\{\max_{0\le j\le i} S_j < 0\}}$ by using the definition of $\alpha_1$ and the fact that $H_1\ge 0$.    By \cite[Lemma  A.2]{cf:CGH2}, $\bbE \left[ \logZ^2\right]< \infty$ (i.e.,  hypothesis  {\bf (T-2)})  implies that for some $\delta_1>0$, \begin{equation}  \label{exponentialetaj}   \bbE[\exp( \delta_1   \widetilde\eta_1)]< \infty. \end{equation}

\noindent Recall that $(\widetilde\eta_\ell, H_\ell-H_{\ell-1})_{\ell\ge1}$ is an IID sequence. By \eqref{exponentialetaj}, we can apply    \cite[Theorem 2.1]{Goldie-Grubel} to obtain some $\delta_2>0$ such that  
 \begin{equation}
 \label{eq:Cdelta_0}
 C_{\delta_2}\,:=\,  \e e^{\delta_2 \sum_{j=1}^\infty \widetilde\eta_j  \, e^{-   H_{j-1}}} \,<\, \infty\,. 
 \end{equation} 
 We deduce from \eqref{eq:Cdelta_0} and  the Markov inequality that for every $\lambda>0$
\begin{equation} 
\p\Bigg(\, \sum_{j=1}^\infty \widetilde\eta_j  \, e^{-   H_{j-1}}  > \lambda\Bigg)
 \, \le\,  C_{\delta_2} e^{-  \delta_2\, \lambda}\, . 
    \label{delta_0}  \end{equation}
In particular,
\begin{equation} \p\Bigg(\sum_{j=1}^\infty \widetilde\eta_j  \, e^{-  H_{j-1}}> e^{x/2}\Bigg) \, \le\,  C_{\delta_2} e^{-\delta_2\, e^{x/2}}. \label{eq:JYkIJ}
   \end{equation}

It remains to estimate $\p\left( H_m > \frac{x}{4}\right)$ according to whether the hypothesis {\bf (R-2)} or {\bf (R-2$^\prime$)} applies. We start with the case {\bf (R-2)}. As discussed right after \cref{lem:Rogozin}, hypothesis  {\bf (T-1)} of \cref{thm:stationary_asymptotics} 
yields that the $\kappa=\xi-1$ moment of $H_1$ is finite, hence, using the Rosenthal inequality \cite{cf:IS}, there exists a constant $C_\kappa>0$ such that  
\begin{equation}\e\left[H_n^\kappa\right]\, \le\, C_\kappa\,  n^\kappa, \qquad \text{ for } n=1,2, \ldots. \label{ehn-kappa}\end{equation}
  \noindent Consequently, using the Markov inequality, for any $x>0$ and $m\ge 1$
  \begin{equation}  
  \p\left( H_m > \frac{x}{4}\right) \le  4^\kappa\, C_\kappa\, m^\kappa\, x^{-\kappa}, \qquad \mbox{in the case {\bf (R-2)}}. \label{eq:IYkIJ-1} 
  \end{equation}
 
 In the case {\bf (R-2$^\prime$)}, $\e[e^{c H_m}]= e^{c_H \, m} $ with $c_H:= \log \e[e^{c H_1}] \in (0,\infty)$. The Markov inequality yields that for any $x\ge 0$ and $m\ge 1$, \begin{equation} \p\left( H_m > \frac{x}{4}\right)  \le   e^{c_H \, m -c x/4}, \qquad \mbox{in the case {\bf(R-2$^\prime$)}}. \label{eq:IYkIJ-2} \end{equation}
  
  Applying \eqref{eq:JYkIJ}, \eqref{eq:IYkIJ-1} and \eqref{eq:IYkIJ-2} to \eqref{YkIJ}, we get that for all  $x\ge 4$,  $\theta \in [0, x/2]$ and $m\ge 1$,

  \begin{equation} \p(J_m(\theta)>x) 
 =\, 
 \begin{cases}
    O_{\mathtt{u}} (m^{\kappa} x^{-\kappa}), \qquad & \mbox{in the case {\bf (R-2)}}, 
 \\
 \,   O_{\mathtt{u}} (e^{c_H m-c x/4}), \qquad & \mbox{in the case {\bf (R-2$^\prime$)}},
 \end{cases}   \label{PJmthetagex} \end{equation}  
 so
  \begin{equation}  F_m(\theta,x) 
 =\, 
 \begin{cases}
 O_{\mathtt{u}}(   m^{\kappa+1} x^{-\kappa}), \qquad & \mbox{in the case {\bf (R-2)}}, 
 \\
 O_{\mathtt{u}} ( m\,   e^{c_H m-c x/4}) , \qquad & \mbox{in the case {\bf (R-2$^\prime$)}}.
 \end{cases}   \label{Fmtheta} \end{equation}  
 
%
%

\subsubsection{Estimate of $\e [ R(y- J_m(\theta))- R(x-J_m(\theta)- t)]$} 
We note that, by \eqref{R(x)2}, $R(x)=O_{\mathtt{u}}(x)$ for $x \ge 1$, and  we observe that $R$ is nondecreasing. 

Let $x\ge 2, y\in[x, 2x]$, $m\ge 1$ and $t\in(0,1)$. Whenever $J_m(\theta)\le x/2$, we have, as a consequence of Lemma~\ref{lem:Rogozin}:
\begin{equation}R(y- J_m(\theta))- R(x-J_m(\theta)- t) 
= c_R (y-x+t) 
+O_{\mathtt{u}}\left(\varphi\left(\frac{x}{2}-1\right)\right),\end{equation}
with the function $\varphi$ defined in \eqref{Rphi}.
Hence, for $x\ge 2, y\in[x, 2x]$, $m\ge 1$ and $t\in(0,1)$
\begin{align}  
& \e \left[R(y- J_m(\theta))- R(x-J_m(\theta)- t) \right]  \nonumber
\\
& \qquad =\, 
\e \left[\left(c_R (y-x+t) +O_{\mathtt{u}}\left(\varphi\left(\frac{x}{2}-1\right)\right)\right) \ind_{\{J_m(\theta)\le \frac{x}2\}}\right] +  \e \left[ O_{\mathtt{u}}(x) \ind_{\{J_m(\theta)> \frac{x}2\}}\right]  \nonumber
\\
& \qquad =\, c_R (y-x+t)
+O_{\mathtt{u}}\left(\varphi\left(\frac{x}{2}-1\right)\right)
+ O_{\mathtt{u}}\left( x \p\left(J_m(\theta)> \frac{x}2\right)\right). \label{x-Ymtheta}
\end{align}

Using \eqref{PJmthetagex}, we conclude that for all for $x\ge 8, y\in[x, 2x]$, $\theta\in[0, {x}/2]$, $m\ge 1$ and $t\in(0,1)$
 \begin{multline} \e [ R(y- J_m(\theta))- R(x-J_m(\theta)- t)] \,=\\
  c_R (y-x+t)+   \begin{cases}    
O_{\mathtt{u}}\left( x^{-(\kappa-2)} + m^\kappa x^{-(\kappa-1)}\right),   & \mbox{in the case {\bf (R-2)}}, \\
O_{\mathtt{u}}\left(  e^{c_H m -\delta_3\, x  }\right),   & \mbox{in the case {\bf (R-2$^\prime$)}},
\end{cases}   \label{RYtheta} \end{multline}
where $\delta_3$ is an arbitrarily fixed constant, satisfying $0<\delta_3<\min(c/8, C_\varphi')$.

\subsubsection{Estimate of $\e [ R(y- J_m(\theta)-t)- R(x-J_m(\theta))]$}

Exactly in the same manner, we obtain that for all $x\ge 8, y\in[x, 2x]$, $\theta\in[0, {x}/2]$, $m\ge 1$ and $t\in(0,1)$
 \begin{multline} \e [ R(y- J_m(\theta)-t)- R(x-J_m(\theta))] \,=\\
  c_R (y-x-t)+   \begin{cases}    
O_{\mathtt{u}}\left( x^{-(\kappa-2)} + m^\kappa x^{-(\kappa-1)}\right),   & \mbox{in the case {\bf (R-2)}}, \\
O_{\mathtt{u}}\left(  e^{c_H m -\delta_3\, x  }\right),   & \mbox{in the case {\bf (R-2$^\prime$)}}. 
\end{cases}   \label{RYtheta'} \end{multline}

\subsubsection{Estimate of $\e[\Lambda(2x, e^{J_m(\theta)}t)]$} 
Recall \eqref{eq:defGamma} for the definition of $ \Lambda(\cdot, \cdot)$. For any $n\ge 1$ and $s>0$, we have:
\begin{equation} \Lambda(x, s)      
     \le \sum_{k=n}^\infty \p[H_k\le x] + n   \p\Bigg(  \sum_{j=1}^\infty \widetilde\eta_j  \, e^{-   H_{j-1}} > s \Bigg).
     \label{Upsilon1-0}
       \end{equation}
  The second term in \eqref{Upsilon1-0} is controlled using \eqref{delta_0}: it is bounded by $C_{\delta_2}\,n\, e^{-  \delta_2 s}$.      
For the first term in \eqref{Upsilon1-0} observe that it is equal to 
\begin{equation}\e [R(x-H_n)]\le R(x) \p(H_n \le x)= O_{\mathtt{u}}(x \p(H_n \le x)),\end{equation}
assuming that $x\ge 1$. 
We complete this estimate by observing that, since
   $\e[H_1]>0$, by  the classical Cram\'er large deviation theorem  there exist $C'>0$ and some small $\delta_4>0$ such that    
   \begin{equation} \label{eq:fromCramer0}
   \p\left(H_k \le \delta_4 k\right) \,\le\, C' e^{-\delta_4 k}\qquad \text{for  } k =1,2, \ldots .       
   \end{equation} 
Consequently, choosing $n : =\lceil  x/\delta_4\rceil$ in \eqref{Upsilon1-0}, we obtain, for $x\ge 1$ and $s>0$,
\begin{equation}\label{eq:boundGamma}
\Lambda(x, s) = O_{\mathtt{u}}( x (e^{-x} + e^{-\delta_2 s})).
\end{equation}

Furthermore, for $\theta\ge 0$, by using the fact that $J_m(\theta)\ge H_m$ and by using \eqref{eq:fromCramer0} again, we have
 \begin{equation}
  \e\Big[e^{-  \delta_2 t e^{ J_m(\theta)}} \Big] 
 \le  \e\Big[ e^{-  \delta_2 t e^{ H_m}}\Big] 
 \le e^{- \delta_2 t e^{\delta_4 m}} +C' e^{-\delta_4 m},\end{equation} 
 which, combined with \eqref{eq:boundGamma}, shows that for all $x\ge 1$, $y\in[x, 2x]$, $\theta\ge 0$, $m\ge 1$
and $t\in (0,1)$,
   \begin{equation}  
    \e\big[\Lambda\big(2x, e^{J_m(\theta)}t\big) \big]
    \,=\, O_{\mathtt{u}} 
    \left(x  \left(e^{-  2 x} +  e^{- \delta_2 t e^{\delta_4 m}} +e^{-\delta_4 m}\right) \right) \,. \label{Upsilon<}
    \end{equation}


%
%
%
%
 
 \subsubsection{Conclusion of the proof of Proposition~\ref{P:F}}
 Now we are ready to give the proof of \eqref{F-asymp}. 
 We make our choice for $m$ and $t$.
 In the case {\bf (R-2)}, we take $m:=\lfloor x^{\delta}\rfloor$ with some $ \delta \in (0, {1}/{\kappa})$.
 In the case {\bf (R-2$^\prime$)}, we take $m:= \lfloor{\delta_3x }/{(2 c_H)}  \rfloor $ , where we recall the constants $ c_H, \delta_3 $ appeared in \eqref{RYtheta}. In both cases, we set $t:=\exp({-{\delta_4m}/{2}})$.
 
   By Lemmas \ref{lem:MarkovJ} and \ref{lem:boundFR}, together with estimates \eqref{Fmtheta}, \eqref{RYtheta}, \eqref{RYtheta'} and \eqref{Upsilon<}, for all $x\ge 8, y\in[x, 2x]$, $\theta\in[0, {x}/2]$,
   \begin{equation} 
   F(\theta,(x, y]) \, = \,  c_R (y-x) + \begin{cases}    
O_{\mathtt{u}} \left(  x^{-(\kappa-2)}\right)\,,   & \mbox{in the case {\bf (R-2)}}, \\
O_{\mathtt{u}} \left(     e^{-\delta_5\, x  }\right),   & \mbox{in the case {\bf (R-2$^\prime$)}}, 
\end{cases}         \end{equation}
\noindent with $\delta_5:= \, \min(\frac{\delta_3}{2} , \frac{\delta_3 \delta_4}{4c_H}, 1)$, 
and the proof of Proposition \ref{P:F} is complete.
\qed

\subsection{Step 4: technical estimates for the control of the cumulative function of $\nu$.} 
      
      Recall that we have chosen $a>0$ in order to ensure that $\nu((-\infty, a])>0$. Recall \eqref{pi-V}, i.e., that for any $a< x \le  y \le 2x$,  \begin{equation}    \nu((x, y])= \int_{(-\infty, a]} \e_{v}[V_a(Y_1, (x, y]) \ind_{\{Y_1 > a\}} ]\nu(\dd v). 
      \label{nua-x}\end{equation} 
In the representation of $V_a(b, (x, y])$ given in Lemma \ref{l:Theta}, we are going to replace $F$ by its asymptotic expression  in the right-hand side of \eqref{F-asymp}.
To this end, we need to control the error terms, this is the goal of the next lemma.      
      Recall the definition of process $\Theta$ just after \eqref{eq:ladder-times} and the notation $\tau_\Theta$ introduced in  Lemma \ref{l:Theta}. 
      
      \begin{lemma} \label{l:beta} If $a$ is large enough, there exists $C>0$ such that for every $b> a$:
      \begin{equation} 
      {\e_b\left[\tau_\Theta(a)\right] \,\le \,  C  \,  b \,,} \label{Etau} 
	      \end{equation}
      and, for every $x\ge 0$,
      \begin{equation}   
      \e_b \left[  \sum_{j=0}^{\tau_\Theta(a)-1} \ind_{\{\Theta_j\ge x \}}\right]\,  
      \le\,  C\, e^{b-x} .  \label{beta}
      \end{equation} \end{lemma}
    \medskip
      
   \begin{proof}
    Recall \eqref{Thetaj+1}. If $\Theta_j>a$, then 
    \begin{equation}
	    \begin{split}
		    \Theta_{j+1}-\Theta_j
   &= \log ( e^{-\Delta_{j+1}}  +  \eta_{j+1} e^{-\Theta_j}) 
   \\
   &\le \log \big( e^{-\Delta_{j+1}}  +  \eta_{j+1} e^{-a} \big)=:W_{j+1}-W_j.
	    \end{split}
    \end{equation}
    
 \noindent In other words, if we consider a real-valued random walk $(W_j)$  with step distribution given as above and starting from $W_0:=b>a$, then for any $j \le \tau_\Theta(a)$, $\Theta_j \le W_j$, hence 
 $$
 \e_b\left[\tau_\Theta(a)\right]
 \,\le\,
 \e_b\left[\tau_W(a)\right]\,,$$
 where $\tau_W(a):= \inf\{n\ge 1: W_n \le a\}$. Similar to \eqref{exponentialetaj}, $\eta_1$ also has some exponential moments; in particular its expectation is finite. Note that for $a$ large enough 
$$\e_0 [e^{W_1}]= \e [e^{-\Delta_1}  +  \eta_1 e^{-a} ] = \e [e^{-\Delta_1} ] + e^{-a} \e [\eta_1 ]<1,$$
  which yields in particular that $\e_0[W_1]<0$.     
   Then it is well known, see \cite[Theorem 3.6.1]{Gut},  that 
   \begin{equation}\e_b\left[\tau_W(a)\right]\, =\, O_{\mathtt{u}}( 1+b-a)\, ,
   \label{gut:b-a}\end{equation} which   gives \eqref{Etau}.
   To show \eqref{beta}, we use again the comparison   $\Theta_j \le W_j$ for all $j \le \tau_\Theta(a)$:  
 \begin{equation*} \e_b \Bigg[   \sum_{j=0}^{\tau_\Theta(a)-1} \ind_{\{\Theta_j\ge x \}}\Bigg]   
 = \sum_{j=0}^\infty \p_b(j < \tau_\Theta(a),  \Theta_j\ge x) 
 \le
 \sum_{j=0}^\infty \p_b(W_j\ge x).
   \end{equation*}
  By the Markov bound, $ \p_b(W_j \ge  x)=\p_0(W_j\ge x-b) \le e^{b-x} (\e_0[e^{W_1}])^j$. Taking  the sum over $j$ yields \eqref{beta} and the proof of Lemma~\ref{l:beta} is complete.
  \end{proof}

   \medskip
We are now ready for:

\begin{proof}[Proof of \cref{thm:stationary_asymptotics}]
We have already remarked  that  {\bf (T-1)} of \cref{thm:stationary_asymptotics} yields {\bf (R-2)} of \cref{lem:Rogozin} by \cite{cf:Doney80}, with $\kappa=\xi-1$. Moreover
	the condition {\bf (H-1)} of \cref{th:main} -- one of the  hypotheses in \cref{thm:stationary_asymptotics} --
	 implies  {\bf (R-1)}  by the following argument. By {\bf (H-1)},  $S_{n_0}$ has a density  with respect to the Lebesgue measure, hence
	 $S_k$, $k\ge n_0$ has a density too. Recall the definitions \eqref{eq:ladder-alpha} and \eqref{def-R}: if $H_{n_0}$ is singular with respect to the Lebesgue measure there exists a Borel set $B\subset [0, \infty)$ of zero Lebesgue measure and full measure under the law of $H_{n_0}$, that is $1=\bbP\left( H_{n_0} \in [0, \infty)\right)= \bbP\left( H_{n_0} \in B\right)$.
But $ \bbP\left( H_{n_0} \in B\right)= \sum_{k \ge n_0} \bbP\left( S_{k} \in B, \, \ga_{n_0}=k\right)\le \sum_{k \ge n_0} \bbP\left( S_{k} \in B\right)=0$. We have reached a contradiction, hence $H_{n_0}$ is not singular. Hence {\bf (R-1)} holds and we can use the conclusions of  \cref{lem:Rogozin}.

	Assume that $a$ is large enough so that Lemma \ref{l:beta} holds. By  \eqref{pi-V}, for every $a<x\le y \le 2x$
	  \begin{equation}
	   \label{nuax}
	  \begin{split}    
	  \nu((x, y])\,&=\, \int_{(-\infty, a]}  \e_v \big[  V_{a}(Y_1, (x,y]) \ind_{\{Y_1 > a\}} \big] \nu(\dd v) 
	  .
	\end{split}
	\end{equation}
We begin with some estimates on $V_{a}(b, (x, y])$. Recall  \eqref{V(x)} for its definition. The first estimate will be useful for large $b$.
Using the fact that $J_k(\theta)\ge H_k+\theta$, we see that for any $x\ge 1$, $y\in[x, 2x]$, $\theta\ge 0$, 
\begin{equation}\label{eq:contrFlarge}
0 \le F(\theta, (x, y])\le R(y-\theta) \leq R(2x)= O_{\mathtt{u}}(x),
\end{equation}
and thus, using \eqref{Etau},
\begin{equation}\label{eq:contrVblarge}
V_a(b, (x, y]) = O_{\mathtt{u}}(x \e_b[\tau_\Theta (a)] )= O_{\mathtt{u}}(xb)\, .
\end{equation} 

Now let us focus on $b$ small, using Proposition \ref{P:F}. In fact we  claim that uniformly in $a < b \le x/4, y\in[x, 2x]$ \begin{equation}
 \label{Vx-asymp}
 V_a(b, (x, y])\,=\, 
     c_R (y-x) \e_b \left[\tau_\Theta(a)\right]  +  O_{\mathtt{u}}\left(b\,  x^{-(\kappa-2)}\right).
     \end{equation}

To show \eqref{Vx-asymp} we use, whenever $\Theta_{j}\le x/2$, the asymptotic development of $F(\Theta_{j}, (x, y])$ given by \eqref{F-asymp}. We have:
\begin{multline}
V_a(b, (x, y])
\, =\\
 \e_b \Bigg[  \sum_{j=0}^{\tau_\Theta(a)-1} \left( c_R (y-x) +  O_{\mathtt{u}}\left(\,  x^{-(\kappa-2)}\right) \right) \ind_{\{\Theta_{j} \le \frac{x}{2}\}}  \Bigg]+ \e_b  \Bigg[ \sum_{j=0}^{\tau_\Theta(a)-1} F(\Theta_{j}, (x,y]) \ind_{\{\Theta_j > \frac{x}{2}\}}  \Bigg] \\
\\ =\,  \e_b  \Bigg[  \sum_{j=0}^{\tau_\Theta(a)-1} \left( c_R (y-x) +  O_{\mathtt{u}}\left(\,  x^{-(\kappa-2)}\right) \right)  \Bigg]+ \e_b  \Bigg[  \sum_{j=0}^{\tau_\Theta(a)-1} O_{\mathtt{u}}(x) \ind_{\{\Theta_{j} > \frac{x}2 \}} \Bigg],
\end{multline}
where we used \eqref{eq:contrFlarge} for the second equality.
Using \eqref{Etau}, the first term is equal to $c_R (y-x)  \e_b \left[  \tau_\Theta(a)\right] +  O_{\mathtt{u}}\left( b \,  x^{-(\kappa-2)}\right)$, while \eqref{beta} yields that the second term is $O_{\mathtt{u}}(x e^{b-x/2})$, hence negligible (since $b\le x/4$). Thus, \eqref{Vx-asymp} is established.

\smallskip

Going back to the estimation of \eqref{nuax}, we use \eqref{Vx-asymp} whenever $a < Y_1 \le x/4$ and then \eqref{eq:contrVblarge} to obtain that for $x\ge 1$, uniformly in $y\in[x, 2x]$ and in $v\in (-\infty, a]$
 \begin{multline}  
 \e_{v}[V_a(Y_1, (x,y])\ind_{\{Y_1 > a\}}]  \,=\\
       \e_{v}\left[ \left( c_R(y-x)  \e_{Y_1}[  \tau_\Theta(a) ] +  O_{\mathtt{u}}\left( Y_1\,  x^{-(\kappa-2)}\right) \right)  \ind_{\{a < Y_1 \le x/4\}}\right]      +\e_{v}[  V_a(Y_1, (x,y])  \ind_{\{Y_1 > x/4\}}  ]  \\    
  = \,
       \e_{v}\left[ \left( c_R(y-x)  \e_{Y_1}[  \tau_\Theta(a) ] +  O_{\mathtt{u}}\left( Y_1\,  x^{-(\kappa-2)}\right) \right) \ind_{\{ Y_1 > a\}} \right]  +\e_{v}[  O_{\mathtt{u}}( Y_1 x)  \ind_{\{Y_1 > x/4\}}  ]  \\      
   = \,  c_R(y-x)  \e_{v}\left[ \e_{Y_1}[  \tau_\Theta(a) ] \ind_{\{ Y_1 > a\}} \right] +  O_{\mathtt{u}}\left( \e_{v}[Y_1\ind_{\{ Y_1 > a\}}] \,  x^{-(\kappa-2)}\right)   + O_{\mathtt{u}}\left( x \e_{v}[ Y_1   \ind_{\{Y_1 > x/4\}}] \right)\, . 
 \label{eq:C_a-bound} 
  \end{multline}
Finally let us argue that, uniformly over $v\in (-\infty,  a]$,  $\e_{v}[Y_1 \ind_{\{ Y_1 > a\}}]$ is bounded and $\e_{v}[ Y_1  \ind_{\{Y_1\ge x/4\}}]=O(x^{-\kappa})$. 
Using the assumption {\bf (T-1)}, $\e[(\logZ^+)^{\kappa+1}]< \infty$ (recall that $\kappa=\xi-1$) and using the fact that under $\p_v$, $Y_1= \logZ_1+ h(v) \le \logZ_1+h(a)$ if $v\le a$, we see that the $\kappa+1$ moment of $Y_1$ is bounded (by a constant that depends on $a$). Hence the claim.
 
 \smallskip

Integrating now over $(-\infty, a]$ against the measure $\nu$ yields that, as $x\to\infty$, uniformly in $y\in[x, 2x]$
\begin{equation} 
\nu((x, y])
\,=\,c_\nu (y-x)+ O_{\mathtt{u}}\left(x^{-(\kappa-2)}\right), \label{Upsilon1}  
\end{equation}
with 
\begin{equation} 
	    c_\nu\, =\,
	    c_R \, \int_{(-\infty, a]} \e_{v}\left[ \e_{Y_1}[  \tau_\Theta(a)   ] \ind_{\{ Y_1 > a\}} \right]  \nu(\dd v) \, . \label{eq:c_nu}
\end{equation}

\medskip

If we know that $\logZ_+$ is exponentially integrable, thus superseding {\bf (T-1)}, we get the claimed improved control on the error by following the same argument we have just developed. The crucial changes are:
\medskip

\begin{enumerate}[leftmargin=0.7 cm]
\item The  main theorem in \cite{cf:Doney80} yields that $\e [\exp( c \logZ^+)]< \infty$ implies 
$\e [\exp( c H_1)]< \infty$.
\item In \eqref{Vx-asymp}, the term  $O_{\mathtt{u}}\left(b\,  x^{-(\kappa-2)}\right)$ becomes $O_{\mathtt{u}}\left(e^{-\delta_6 x}\right)$, with some positive constant $\delta_6$. This can be done exactly in the same way, by applying the case {\bf(R-2$^\prime$)} of Proposition~\ref{P:F}.
\item The expectation $\e_{v}[ Y_1  \ind_{\{Y_1\ge x/4\}}]$ in \eqref{eq:C_a-bound} is now dominated uniformly over $v\in (-\infty, a]$ by $e^{-\delta_6 x}$, by eventually choosing a smaller constant $\delta_6$, instead of $x^{-\kappa}$.
\end{enumerate}
 \medskip

%
%
%
%
%
 
 The proof of  \cref{thm:stationary_asymptotics} is therefore complete.
\end{proof}

\begin{rem} The formula \eqref{eq:c_nu} provides an expression for $c_\nu$, the constant of the leading term in the right-hand-side of \eqref{eq:stationary_asymptotics}. We can express $d_\nu$, the other constant in \eqref{eq:stationary_asymptotics}, as follows: 
	\begin{equation}   d_\nu =  \nu((-\infty, a])+ \int_{(-\infty, a]}  \e_{v}\left[  \ind_{\{ Y_1 > a\}} \e_{Y_1} \Big(  \sum_{j=0}^{\tau_\Theta(a)-1} \psi(\Theta_{j}) \Big) \right]  \nu(\dd v)  ,
	  \end{equation} where $\psi(\theta):=c_R'-c_R  \e[\log \big(e^{ \theta} + \sum_{k=1}^\infty \widetilde\eta_k  \, e^{-   H_{k-1}}\big)]$, $c_R$ and $c'_R$ are given in \eqref{eq:cRc'R} and $ \widetilde \eta_j$ in \eqref{eq:defetatilde}. Note that ${d_{\nu}}/{c_{\nu}}$ is related to $\kappa_2$ through \eqref{eq:kappa2}. 
   \end{rem}
\section{Contraction properties: proofs}

In this section, we prove \cref{thm:contractive}, that is, we show that the dynamical system defined by $T_\Gamma$ is contractive with a suitably explicit rate.
The techniques we use are lightly adapted from a recent and so far unpublished work by Quentin Moulard \cite{cf:QM}.

Let $\cM_+$ be the set of two-by-two matrices with positive elements and positive determinant, i.e.\ of
\begin{equation}
	A =
	\begin{pmatrix}
		a & b
		\\
		c & d
	\end{pmatrix}
	,
	\label{eq:generic_M}
\end{equation}
with $a,b,c,d,ad-bc>0$.  This includes matrices of the form in \cref{eq:keyM}, so long as $\varepsilon \in (0,1)$, and importantly also arbitrary products of such matrices.
With the components labelled as in \cref{eq:generic_M},
\begin{equation}
	A 
	\begin{pmatrix}
		1 \\ e^x
	\end{pmatrix}
	=
	\begin{pmatrix}
		a + b e^x
		\\
		c + d e^x
	\end{pmatrix}
	=
	(a+b e^x)
	\begin{pmatrix}
		1 \\ e^{\cT_A(x)}
	\end{pmatrix}
	\textup{ with }
	\cT_A(x) 
	:=
	\log\left( \frac{c+d e^x}{a + b e^x} \right),
	\label{eq:TA_def}
\end{equation}
then
\begin{equation}
	\cT'_A(x)
	=
	\frac{ad - bc}{(a e^{-x} + b)(c + d e^{x})}
	\quad \implies \quad
	0 <
	\cT'_A(x)
	<
	\frac{\textrm{det}(A)}{A_{11}A_{22}}
	=:
	\tau(A)\,<\, 1\, ,
	\label{eq:cTA_deriv}
\end{equation}
for all $A \in \cM_+$, $x \in \R$; in particular $\cT_A$ is  increasing, and so invertible on $\cT_A(\R) = (\log \tfrac{c}{a},\log \tfrac{d}{b})$. 
We define for $G$ of bounded total variation 
\begin{equation}
	\cT_A G(x)
	:=
	\begin{cases}
		\lim_{y \to -\infty} G(y)
		, &
		x \le \log (c/a)
		\\
		G \circ \cT_A^{-1} (x)
		, &
		\log (c/a) < x < \log (d/b)
		\\
		\lim_{y \to \infty} G(y)
		, &
		x \ge \log (d/b)
	\end{cases}
	\label{eq:TA_G_action}
\end{equation}
($\lim_{y \to \pm \infty} G(y)$ exist because $G$ is the difference of two nonincreasing functions). Observe that if $G$ is the risk function of some variable $X$, then $\cT_A G=G_{\cT_A (X)}$.
Furthermore, note that when $M=M(e^{-\Gamma}, e^{-\mathtt{z}})$, we have $\cT_M (x) = \mathtt{z} + h_\Gamma (x)$, so $T_\Gamma$ introduced in \eqref{eq:TGdef}
can be expressed as follows: when $G$ is the risk function of some variable $X$, we have
\begin{equation}
	T_{\Gamma} G(x) 
	\,=\,
	\bbE[\cT_M G(x)]\,
	,
	\label{eq:TG_matrix}
\end{equation}
and this extends to the case when $G$ is the difference of the risk functions of two random variables. 
It is obvious from \cref{eq:TA_def} that $\cT_A \circ \cT_B = \cT_{AB}$, as a result of which
\begin{equation}
	(T_\Gamma)^n G(x)
	=
	\bbE\left[ \cT_{M_1 \dots M_n} G(x) \right].
	\label{eq:mat_G_power_n}
\end{equation}
\medskip

\begin{lemma}
Let $G_j$ be the risk function of $X_j \in \bbL^1$, $j=1,2$. Then for $A \in \cM_+$
we have that $\cT_A G_j$, $j=1,2$, are the risk functions of two $\bbL^1$ random variables, in fact $\bbL^\infty$, and 
\begin{equation}
		\| \cT_A G_1-  \cT_A G_2 \|_1
		\le
		\tau(A) \| G_1-G_2\|_1
		.
		\label{eq:tau_A_bound'}
	\end{equation}
	\label{lem:tau_A_bound}
\end{lemma}

\medskip

\begin{proof}
	We use a well-known identity: the $\bbL^1$ distance between the risk functions (or cumulative functions) of two real random variables is equal to the Wasserstein distance between the laws of these variables, 
	in our case: $\|G_1-G_2\|_1=d_{W, 1}(X_1, X_2)$, where with a slight abuse of notation, we write $d_{W, 1}(X_1, X_2)$ for the Wasserstein distance between the laws of $X_1$ and of $X_2$. 
	Similarly, $\|\cT_A G_1-\cT_A G_2\|_1=d_{W, 1}(\cT_A(X_1), \cT_A(X_2))$. 
	Hence, our aim is to prove that 
	\begin{equation}
	d_{W, 1}(\cT_A(X_1), \cT_A(X_2))\le \tau(A) d_{W, 1}(X_1, X_2).
	\end{equation} 
	This is immediate using the Taylor inequality
	\begin{equation}
	\forall x, y \in \R, \quad |\cT_A(x) - \cT_A(y)|\le \tau(A) |x-y|,
	\end{equation}
	which follows from \eqref{eq:cTA_deriv}.
\end{proof}

\medskip
Recall \eqref{eq:TG0_def} for $T_{\Gamma,0}G$. 
Combining the above lemma with \cref{eq:mat_G_power_n}, we have
\medskip

\begin{corollary}
	\label{lem:TN_by_tau}
	For any $n \in \bN$ and any $G$ the difference of the risk functions of two $\bbL^1$ random variables,
	\begin{equation}
		\| T_{\Gamma,0}^n G \|_1
		\le
		\bbE \left[ \tau(M_1 \dots M_n) \right] \|G\|_1 \, .
		\label{eq:Tn_by_tau}
	\end{equation}
\end{corollary}

\medskip

Consequently, we can relate \cref{thm:contractive} to a bound on $\sum_{n=0}^\infty\bbE[\tau(M_1 \dots M_n)]$.  We obtain such a bound as follows: set
\begin{equation}
n_{\gG, \gs}\, :=\, \left \lceil 37 \, \left( \gG/\gs \right)^2\right\rceil\, . 
\end{equation}

\medskip

\begin{lemma}
\label{th:tau_G2_bound}
	There exists $\Gamma_0 < \infty$ such that
	\begin{equation}
		\bbE \left[\tau\left(M_1 \dots M_{n_{\gG, \gs}}\right)\right]
		\,\le \,  \frac 12 \, ,
		\label{eq:tau_G2_bound}
	\end{equation}
	for all $\Gamma \ge \Gamma_0$.
	\label{lem:tau_G2_bound}
\end{lemma}
\medskip

\begin{proof}
	Note that $\tau(A)=\tau(\lambda A)$ for any $A \in \cM_+,\ \lambda >0$, so instead of $M$ in \eqref{eq:keyM} we can equivalently consider
	\begin{equation}
		Q_j
		:=
		\begin{pmatrix}
			e^{-\mathtt{z}_j/2}
			&
			e^{-\Gamma - \mathtt z_j / 2}
			\\
			e^{-\Gamma + \mathtt z_j / 2}
			&
			e^{\mathtt{z}_j/2}
		\end{pmatrix}
		.
		\label{eq:Qj_def}
	\end{equation}
	Abbreviating $\tilde Q^{(n)} = Q_1 \dots Q_n$, we have
	\begin{equation}
		\tau(M_1 \dots M_n)
		\,=\,
		\tau\left(\tilde Q^{(n)}\right)
		\,\le\, 
		1\Big/
		\left[ \tilde Q^{(n)}_{11} \tilde Q^{(n)}_{22} \right]
		.
		\label{eq:tau_Q}
	\end{equation}
	Since all of the matrices involved are positive, we can bound the matrix products involved from below by single products of elements; most obviously
	\begin{equation}
		\tilde Q^{(n)}_{11}
		\tilde Q^{(n)}_{22}
		\ge 
		(Q_1)_{11} \dots (Q_n)_{11} \times (Q_1)_{22} \dots(Q_n)_{22}
		=
		1
		,
	\end{equation}
	but also
	\begin{equation}
		\tilde Q^{(n)}_{11}
		\tilde Q^{(n)}_{22}
		\ge 
		(Q_1)_{11} \dots (Q_n)_{11}
		\times
		(Q_1)_{21} (Q_2)_{11} \dots (Q_{n-1})_{11}(Q_{n})_{12}
		=
		e^{-2\Gamma} e^{- \sum_{j=2}^n \mathtt z_j}
	\end{equation}
	and
	\begin{equation}
		\tilde Q^{(n)}_{11}
		\tilde Q^{(n)}_{22}
		\ge 
		(Q_1)_{12} (Q_2)_{22} \dots (Q_{n-1})_{22}(Q_{n})_{21}
		\times
		(Q_1)_{22} \dots (Q_n)_{22}
		=
		e^{-2\Gamma} e^{\sum_{j=2}^n \mathtt z_j}
		.
	\end{equation}
	Assembling these three estimates and combining them with \cref{eq:tau_Q},
	\begin{equation}
		\tau(M_1 \dots M_n)
		\le
		\min\left( 
			1,
			e^{2 \Gamma} e^{\sum_{j=2}^n \mathtt z_j}
			,
			e^{2 \Gamma} e^{-\sum_{j=2}^n \mathtt z_j}
		\right)
		=
		\min \left( 
			1,
			e^{2 \Gamma} e^{- |\sum_{j=2}^n \mathtt z_j|}
		\right)
		.
		\label{eq:tau_Q_usable}
	\end{equation}

	Since the $\mathtt z_j\in \bbL^2$ are centered and IID, we can use the Central Limit Theorem together with 
	$3\gG /(\gs \sqrt{n_{\gG, \gs} })\le 1/2$ for $\gG$ sufficiently large
	 to see that with $\cN$ denoting a standard Gaussian variable
	\begin{equation}
		\bbP \left( \bigg |\sum_{j=2}^{n_{\gG, \gs}} \mathtt z_j \bigg| \,\ge\, 3 \,\gG \right)\,
		\ge \, \bbP \left(\vert \cN \vert \ge 1/2\right)\,=:\, \gd\,> \, \frac 35\,.
		\label{eq:z_sum_CLT}
	\end{equation}
Combining this with \cref{eq:tau_Q} we see that by suitably choosing $\gG_0$ we have that for every $\gG\ge \gG_0$
	\begin{equation}
		{\bbE\left[ \tau(M_1 \dots M_{n_{\gG, \gs}} \right]}\,
		\le\, 
		(1-\delta)
		+ \delta e^{-\Gamma}\,
		\le\,
		1 - \delta (1 - e^{-\Gamma_0})
		\, \le \, \frac 12
		.
	\end{equation}
\end{proof}

\begin{proof}[Proof of \cref{thm:contractive}]
Combining this with \cref{lem:TN_by_tau}
and 
Lemma~\ref{th:tau_G2_bound}
we obtain
\begin{equation}
		\left\| T_{\Gamma,0}^{n_{\gG, \gs}} G \right\|_1\,
		\le\,
\frac 12\, \|G\|_1\,.
	\end{equation}
Writing $n = n_{\gG, \gs} m + r$ with  $r = n - n_{\gG, \gs}m \in [0, n_{\gG, \gs})$ we have
	\begin{equation}
		T_{\Gamma_0}^n G
		\, =\,
		T_{\Gamma,0}^{n_{\gG, \gs}} \left[ T_\Gamma^{n_{\gG, \gs}}  \dots\left[ T_\Gamma^r G \right] \right];
	\end{equation}
	noting that $\| T_{\Gamma,0}^r G \|_1 \le \| G \|_1$ (as can be seen from either \cref{eq:Tbound0} or \cref{eq:tau_Q_usable}),
	applying \cref{lem:TN_by_tau,lem:tau_G2_bound} iteratively leads to
	\begin{equation}
		\left\| T_{\Gamma,0}^n G \right\|_1
		\,
		\le\,
		\left(\frac 12 \right)^m \| G \|_1\, ,
	\end{equation}
	so that
	\begin{equation}
		\sum_{n=0}^\infty 
		\left\| T_{\Gamma,0}^n G \right\|_1
		\,\le\, n_{\gG, \gs}
		 \sum_{m=0}^\infty \left(\frac 12 \right)^m
		=
		\, 2 \, n_{\gG, \gs}\, \le \, \frac{75}{\gs^2}\, {\gG^2},
	\end{equation}	
	which completes the proof of \cref{thm:contractive}.
\end{proof}

\section{The one-step bound}
\label{sec:onestep}

This section is devoted to the proof of  \Cref{thm:onestep}. 

\begin{proof}
We start by observing that \eqref{eq:Ck} is a direct consequence of \eqref{eq:asymptlar}. Moreover
\eqref{eq:Lyap} follows by the following simple calculation: by \eqref{eq:L_k} 
 and \eqref{eq:DHprobability} we have
\begin{equation}
\label{eq:for-Lyap-1}
L_\gG [G_\gG ]\,=\, \frac 1{C_\gG }\int_0^\infty \frac{F_\rar(\gG -x)}{1+e^{\gG -x}} \dd x+
\int_{-\infty}^0 \frac{G_\gG (x)} {1+e^{\gG -x}} \dd x\, ,
\end{equation}
where for notational brevity, we will write in this proof $G_\gG$ in lieu of $G_{\gamma_\gG}$, and analogous notation for the cumulative function $F_{\gamma_\Gamma}$.  Recall that these were introduced in \cref{eq:DHprobability,eq:cumFk} above. 
Using $G_\gG (x)\in [0, 1]$ we readily see that the second addendum is smaller than $\exp(-\gG )$.
Moreover
\begin{equation}
\int_0^\infty \frac{F_\rar(\gG -x)}{1+e^{\gG -x}} \dd x\, =\, 
\int_{\bbR} \frac{F_\rar(y)}{1+e^{y}} \dd y - \int_\gG ^\infty \frac{F_\rar(y)}{1+e^{y}} \dd y\, =\, 
\int_{\bbR} \frac{F_\rar(y)}{1+e^{y}} \dd y +O\left(\gG  \exp(-\gG )\right)
\, ,
\end{equation}
where in the last step we used \eqref{eq:asymptlar} for $F_\rar$ instead of $F_\lar$.
\medskip

\begin{rem} \label{rem:sym3}
Repeating the same argument using instead the functional $L_\gG^\star [\, \cdot \,]$  defined in \eqref{eq:L_k2}  of Remark~\ref{rem:sym2} one obtains
\begin{equation}
		\label{eq:Lyap2}
		L^\star_\gG [G_{\gamma_\gG} ]\,=\,  \frac 1{C_\gG }\int_\bbR \frac{F_\lar(y)}{1+e^y} \dd y+ O( \exp(-\gG ))\,.
	\end{equation} 
Then, similarly to \eqref{eq:kappa_1_formula-0}, we arrive at another formula for $\kappa_1$: \begin{equation}
	\kappa_1
	=
	\frac12
	\int_\bbR \frac{F_\lar(y)}{1+e^{y}} \dd y
	= \frac 12
	\int \log (1 + e^{-y}) \nu(\dd y)
	\, ,
	\label{eq:kappa_1_formula}
\end{equation}
where the second equality follows from an integration by parts, and we recall that $F_\lar$ is the cumulative function of the (normalized) stationary measure $\nu$ of $Y$. This highlights a symmetry with respect to exchanging  $F_\lar$ and  $F_\rar$.
\end{rem}

\medskip

Now for the proof of \eqref{eq:onestep}.
We refer to assumption  ``{\bf (T-1)} for $\pm \logZ$'' as ``$\logZ\in \bbL^\xi$'' and to ``{\bf (T-1$^\prime$)} for $\pm \logZ$''
as ``exponential tails assumption'' and, when the assumption is omitted, $\logZ\in \bbL^\xi$ is assumed. We have
\begin{equation}
\label{eq:2addenda-2}
\begin{split}
\left\Vert T_\gG G_\gG  - G_\gG  \right \Vert_1\, &=\, 
\left\Vert (T_\gG G_\gG  - G_\gG  )\ind_{(-\infty,0)}\right \Vert_1 + \left\Vert (T_\gG G_\gG  - G_\gG  )\ind_{(0, \infty)}\right \Vert_1
\\
&=\,
\left\Vert (T_\gG F_\gG  - F_\gG  )\ind_{(-\infty,0)}\right \Vert_1 + \left\Vert (T_\gG G_\gG  - G_\gG  )\ind_{(0, \infty)}\right \Vert_1\, ,
\end{split}
\end{equation}

\noindent where {$G_\Gamma := G_{\gamma_\Gamma}$ and $F_\Gamma := G_{\gamma_\Gamma}$ are defined in \eqref{eq:DHprobability} and \eqref{eq:cumFk} respectively.}
If $\zeta$ is symmetric the two terms coincide; in general, they do not coincide but they can be treated 
in the same way because one can be obtained from the other by replacing $\logZ$ with $-\logZ$.  
So we focus on the first one 
and use  \eqref{eq:FXT} to obtain
\begin{equation}
\begin{split}
\left\Vert (T_\gG F_\gG  - F_\gG  )\ind_{(-\infty,0)}\right \Vert_1\, & =\, \left\Vert (F_\logZ (-\gG+\cdot)+T_{\gG,0} F_\gG  - F_\gG  )\ind_{(-\infty,0)}\right \Vert_1
\\
& \le \, \int_{-\infty}^{-\gG} F_\logZ (x)\dd x +\left\Vert \left(T_{\gG,0} F_\gG  - F_\gG  \right)\ind_{(-\infty,0)}\right \Vert_1\,.
\end{split}
\end{equation}
Note that  the first addendum in the right-most term is $O(\gG^{-\xi+1})$  when $\logZ\in \bbL^\xi$ (and $O(\exp(-c \gG))$ assuming exponential tails). So we focus on the second one, which is equal to
\begin{equation}
	A_{\gG ,1}\, :=\,\int_{-\infty}^0 \left \vert  \int \ind_{(x-\gG ,x+\gG )}(z) F_\gG(h_\gG ^{-1}(x-z))  \gz( \dd z) - F_\gG(x) \right\vert  \dd x\, . 
\end{equation}

The first step in controlling $A_{\gG ,1}$ is to remark that we can avoid the nuisance of the fact that 
$F_\gG (y)$ has two different expressions according to the sign of $y$. 
Namely, we want to relate it to
\begin{equation}
A_{\gG ,2}\, :=\,\frac 1{C_\gG }\int_{-\infty}^0 \left \vert  \int\ind_{(x-\gG ,x+\gG )}(z) F_\lar(h_\gG ^{-1}(x-z)+\gG )  \gz( \dd z) - F_\lar(x+\gG ) \right\vert  \dd x\, .
\end{equation}
We can do this because 
\begin{equation}
\left\vert A_{\gG ,2}-A_{\gG ,1}\right\vert\, \le \,
\int_{-\infty}^0    \int
\left. \ind_{(x-\gG ,x)}(z) \left \vert \frac{
F_\lar(\gG +y) + F_\rar(\gG -y)}{C_\gG }-1 \right \vert
\right|_{y = h_\gG ^{-1}(x-z)}
\gz(\dd z) \dd x\, ,
\label{eq:A2-A1}
\end{equation}
using the observation that $h_\gG ^{-1}(0)=0$, $h_\gG ^{-1} (\gG ^-)=+\infty$, so $\ind_{(x-\gG ,x+\gG )}(z) \ind_{(0,\infty)}(h^{-1}_\gG (x-z)) = \ind_{(x-\gG ,x)}(z)$.
Furthermore by \cref{eq:asymptlar,eq:Ck}  we have
\begin{equation}
\left \vert \frac{F_\lar(\gG +y) + F_\rar(\gG -y)}{C_\gG }-1 \right \vert \,\le\, C \times  \begin{cases}
	\gG ^{-\xi+2} 
	& \text{ if } y\in [0, \gG /2] \, ,
\\
y/\gG  & \text{ if } y >\gG  /2\, ,
\end{cases}
\end{equation}
for suitably chosen $C>0$ and for $\gG $ sufficiently large (in the exponential case the first line of the estimate may be replaced by 
$\exp(-\gd \gG/2)/\gG$). 
Combining this with \cref{eq:A2-A1}, we see that
 $\vert A_{\gG ,2}-A_{\gG ,1}\vert$ is bounded above by a constant times 
\begin{equation}
	\label{eq:interm2.k}
	\begin{split}
		&
		\gG ^{-\xi+2}
		\int_{-\infty}^0 \int \ind_{(x-\gG ,x)}(z) \zeta(\dd z) \dd x
		+
		\gG^{-1} 
		\int_{-\infty}^0 \int \ind_{(x-\gG ,x-h_\gG (\gG /2))}(z) h^{-1}_\gG  (x-z) \zeta(\dd z) \dd x
		\\ &
		\le
		\gG ^{-\xi+2}
		\int_{-\infty}^0 F_\zeta(x) \dd x
		+
		\gG^{-1} 
		F_\zeta(-h_\gG (\gG /2)) 
		\int_\bbR \vert y\vert  h'_\gG (y) \dd y,
	\end{split}
\end{equation}
where the second addendum is treated 
by a change of variable of integration from $x$ to $y = h_\gG ^{-1}(x-z)$, taking the absolute value of the integrand and by releasing the integration constraints on $y$.
Note that in the exponential case the prefactors $\gG^{-\xi+2}$ in \eqref{eq:interm2.k} may be replaced by 
 $\exp(-\gd \gG/2) / \gG$.
The assumption on $\logZ\in \bbL^\xi$  
together with the Markov inequality imply $F_\zeta (x) = O(|x|^{-\xi})$ for $x \to -\infty$
 ($F_\zeta (x) = 
 O(\exp(-c\vert x\vert))
 $ in the exponential case). Since $h_\gG (\gG /2) \sim \gG /2$, $h'_\gG (y)<1$ and the fact that $h'_\gG(y)$ decays exponentially for $y >\gG $ (so $\int \vert y\vert  h'_\gG (y) \dd y=O(\gG ^2)$), we have $\vert A_{\gG ,2}-A_{\gG ,1}\vert = O(\gG^{-\xi+2})+
O(\gG^{-\xi+1})$ (in the exponential case: $ O(\exp(-\gd \gG/2) / \gG)+
O(\gG \exp(- c\gG/2)  )$), so
\begin{equation}
	\left\vert A_{\gG ,2}-A_{\gG ,1}\right\vert\, 
	=\begin{cases}
	O\left( \gG ^{-\xi+2} \right) & \text{ if } \logZ \in \bbL^\xi, \\
	O(\gG \exp(-(\gd \wedge c) \gG/2)) & \text{ under exponential tails assumption}.
	\end{cases}
\end{equation}
 
Having shown this, we now need to control $A_{\gG ,2}$.  To do this we introduce
\begin{equation}
\hup _\gG (x)\, :=\, h(x+\gG )-\gG \, 
\end{equation}
which approximates $h_\Gamma$ for a certain range of arguments in the sense that
 \begin{equation}
\label{eq:ident-before-asympt1}
0 \, \le \, \hup _\gG (x)- h_\gG (x)\, =\, \log(1+ \exp(x-\gG ))\,= \begin{cases}
O(\exp(x-\gG ))& \text{ for } x-\gG \to -\infty\, , \\
O(x-\gG ) & \text{ for } x-\gG \to +\infty\, ,
\end{cases}
\end{equation}
\begin{equation}
\label{eq:ident-before-asympt2}
0 \, \le \, \hup '_\gG (x)- h'_\gG (x)\, =\, 1/(1+ \exp(\gG -x))\, = \begin{cases}
O(\exp(x-\gG ))& \text{ for } x-\gG \to -\infty\, , \\
\le 1 & \text{ for every } x \text{ and } \gG \, ,
\end{cases}
\end{equation}
and
\begin{equation}
\label{eq:hgG-1b}
0 \le h_\gG ^{-1}(y)-  \hup _\gG^{-1}(y)\,=\, - \log \left( 1- e^{y-\gG}\right)\, \qquad \forall\, y\in (-\Gamma, \Gamma).
\end{equation}
Moreover we point out that $h_\gG '(x)$ and $\hup '_\gG (x)$ are in $(0,1)$ for every $x$ and, as shown in Figure~\ref{fig:1}, the two functions almost coincide for $x\ll \gG $. They start to differ when $x$ approaches $\gG $ from the left, where  $\hup '_\gG (x)$ keeps being very close to one while $\hup _\gG (x)\sim x$. 
 
Using the characterization of $F_\lar(\cdot)$ as the cumulative function of a stationary measure,
\begin{multline}
	\label{eq:passo-interm1}
	\int\ind_{(-\infty,x+\gG )}(z) F_\lar(\hup _\gG ^{-1}(x-z)+\gG ) \gz(\dd z) - F_\lar(x+\gG )\,=\\
	\int \ind_{(-\infty,x)}(z) F_\lar(h^{-1}(x-z)) \gz(\dd z) - F_\lar(x)\, =\,0 ,
\end{multline}
and so
\begin{multline}
		A_{\gG ,2}\, 
		\,=
		\\ 
		\frac 1{C_\gG } 
		\int_{-\infty}^0  \left\vert 
		\int \left[
			\ind_{(x-\gG ,x+\gG )}(z) F_\lar(h_\gG ^{-1}(x-z)+ \gG) 
			-
			\ind_{(-\infty,x+\gG )}(z) F_\lar(\hup _\gG ^{-1}(x-z)+ \gG) 
		\right] \zeta(\dd z)
		\right\vert \dd x
		\\ 
		\le 
		\frac 1{C_\gG } 
		\int_{-\infty}^0 \int
		\bigg\{ 
			\ind_{(-\infty,x-\gG/2 ]}(z) F_{\lar}(\hup _\gG ^{-1}(x-z)+ \gG)
			+ \ind _{(x-\gG, x-\gG/2]}(z)  F_\lar(h_\gG ^{-1}(x-z)+ \gG) \\
			+\ind_{(x-\gG/2 ,x+\gG )}(z) \Big[ F_\lar(h_\gG ^{-1}(x-z)+ \gG)-  F_\lar(\hup _\gG ^{-1}(x-z)+ \gG) \Big]
		\bigg\} \zeta(\dd z) \dd x
		\\ =:\, T_1+T_2+T_3\,
		.
	\label{eq:Ak2_main}
\end{multline}
For $T_1$
we first note that, since $\hup ^{-1}_\gG (x') \sim x'$ and $F_\lar(x')\sim x'$ for  $x' \to +\infty$, we have for $\gG$ large
\begin{equation}
	\label{eq:T1_prelim}
	\begin{split}
		\ind_{(-\infty,x-\gG/2 ]}(z) F_{\lar}(\hup _\gG ^{-1}(x-z)+\gG)
		\,&=\,
		\ind_{[\gG/2 ,\infty)}(x-z) F_\lar(\hup _\gG ^{-1}(x-z)+\gG)
		\\
		&\le\, 
		2 (x-z+ \gG) \ind_{[\gG/2 ,\infty)}(x-z)
		,
	\end{split}
\end{equation}
and integrating by parts (keep in mind that $x<0$)
\begin{equation}
	\int (x-z) \ind_{[\gG/2 ,\infty)}(x-z) \zeta (\dd z)
	=
	\frac \gG 2 F_\zeta(x-\gG )
	+
	\int_{-\infty}^{x-\gG/2 }
	F_\zeta (z) \dd z
	=
	O\left( |x-\gG |^{-\xi+1} \right)\, ,
	\label{eq:T1_intermediate}
\end{equation}
where we have used  that $\logZ\in \bbL^{\xi}$; in the exponential case the estimate is $O(\exp(-c \vert x-\gG\vert/2)$.
	Regarding the remaining term in \cref{eq:T1_prelim}, note that
\begin{equation}
 \gG \int \ind_{[\gG/2 ,\infty)}(x-z) \zeta (\dd z)\, =\, \gG  F_\zeta(x-\gG/2 )\, ,
\end{equation}
and we can bound this by the Markov inequality, combining with \cref{eq:T1_intermediate} to give
\begin{equation}
	\int (x-z+ \gG) \ind_{[\gG/2 ,\infty)}(x-z) \zeta (\dd z)\, =\, O\left( |x-\gG |^{-\xi +1} \right)\, ,
	\end{equation}
for $\logZ\in \bbL^\xi$
	and  $O(\vert x-\gG\vert\exp(-c \vert x-\gG\vert/2))$ in the exponential case.
Therefore integrating in $x$ and dividing by $C_\gG $ we obtain $T_1=O(\gG ^{-\xi +1})$ and $O(\exp(-c \gG/3))$ in the exponential case.
\medskip 

For $T_2$ we use $F_\lar(x')\sim x'$ for $x'\to \infty$ (note that $x-z \in [\gG/2, \gG)$) so
\begin{equation}
  F_\lar(h_\gG ^{-1}(x-z)+ \gG) \, \le \, 2 (h_\gG ^{-1}(x-z)+ \gG)\, .
\end{equation}
Moreover
\begin{equation}
\ind_{(-\infty, 0)}(x)
 \ind _{(x-\gG, x-\gG/2]}(z)\, =\, 
 \ind_{(-\infty, -\gG]}(z)
 \ind _{(z+\gG/2, z+\gG]}(x)+  \ind_{(-\gG, -\gG/2)}(z)
  \ind _{(z+\gG/2,0]}(x),
\end{equation}
and we remark that
\begin{equation}
\int_{z+\frac \gG 2}^{z+\gG } (h_\gG ^{-1}(x-z)+ \gG) \dd x\, =\,
\int_{\frac \gG 2}^{\gG } \left(h_\gG ^{-1}(x)+ \gG\right) \dd x\, =\, O\left( \gG^2 \right),
\end{equation}
because when $x\nearrow \gG$ we have $ h_\gG ^{-1}(x)-\gG \sim \log (\gG-x)$ and the integral converges.
For the other term we have by using that $-z\in (-\gG/2, \gG)$ and the behavior of $h_\gG^{-1}$ approaching $\gG$
\begin{equation}
\int_{z+\frac \gG 2}^{0} (h_\gG ^{-1}(x-z)+ \gG) \dd x\, =\,
\int_{\frac\gG 2}^{-z} \left(h_\gG ^{-1}(x)+ \gG\right) \dd x
\, \le \, \int_{\frac\gG 2}^{\gG} \left(h_\gG ^{-1}(x)+ \gG\right) \dd x
\, =\, O\left( \gG^2 \right).
\end{equation}
By integrating over $z$, which is required to be smaller that $-\gG/2$ in both terms, and dividing by $C_\gG$ we see that $T_2=O(\gG ^{-\xi +1})$
and $O(\gG \exp(-c\gG/2))$ in the exponential case.
\medskip

For $T_3$ we note that by \eqref{eq:hgG-1b}
\begin{equation}
\label{eq:hgG-1b-2}
0 \le h_\gG ^{-1}(y)-  \hup _\gG^{-1}(y)\,=\,O\left(e^{y-\gG}\right)\, ,
\end{equation}
 when $y\in (-\gG, \gG/2)$.
 We then note that $C_\gG T_3$ is equal to
\begin{multline}
\int_{-\infty}^0 \dd x \int _\bbR \zeta(\dd z) \ind_{(x-\gG/2, x+ \gG)}(z)
\int \ind_{(\hup _\gG ^{-1}(x-z), h_\gG ^{-1}(x-z))}(y-\gG)\nu(\dd y)
\, =\\
\int_{-\infty}^\gG  \zeta(\dd z)  \int _{z-\gG}^{(z+\gG/2)\wedge 0}  \dd x
\int \ind_{(\hup _\gG ^{-1}(x-z), h_\gG ^{-1}(x-z))}(y-\gG)\nu(\dd y)\, .
\end{multline}
We then change variable from $x$ to $x-z$ and ignore the constraint $x \le 0$ to obtain
\begin{equation}
\begin{split}
C_\gG T_3
\, &\le \, 
\int_{-\infty}^\gG  \zeta(\dd z)  \int _{-\gG}^{\gG/2}  \dd x
\int \ind_{(\hup _\gG ^{-1}(x), h_\gG ^{-1}(x))}(y-\gG)\nu(\dd y)
\\
&\le\,   \int _{-\gG}^{\gG/2}  \dd x
\int \ind_{(\hup _\gG ^{-1}(x), h_\gG ^{-1}(x))}(y-\gG)\nu(\dd y)
\\
&=\, \int_{-\infty}^{h_\gG^{-1}(\gG/2)+ \gG} \nu(\dd y) \int_{h_\gG(y-\gG) \wedge (-\gG)}^{\hup_\gG(y-\gG) \vee (\gG/2)} \dd x
\\
&\le \, \int_{-\infty}^{h_\gG^{-1}(\gG/2)+ \gG} \nu(\dd y) \left(\hup_\gG(y-\gG)-h_\gG(y-\gG)\right) \, ,
\end{split}
\end{equation}
and then bounding the last integrand using \cref{eq:hgG-1b-2}, we see that $C_\gG T_3$ is bounded above by a constant times 
\begin{equation}
 \int_{-\infty}^{h_\gG^{-1}(\gG/2)+ \gG}   e^{y-2\gG} \nu(\dd y) \, \le \,
 e^{h_\gG^{-1}(\gG/2)-\gG} \nu \left( (-\infty, h_\gG^{-1}(\gG/2)+ \gG]\right)\, =\, O\left( \gG e^{-\gG/2}\right)
 \,,
\end{equation}
where for the last estimate we observe that  $h_\gG^{-1}(\gG/2) \sim \Gamma/2$ by \cref{eq:almost-id} and use \cref{eq:stationary_asymptotics2} to bound $\nu$.
Therefore $T_3$ is negligible.

\smallskip

By collecting all the estimates, and properly redefining the value of $\gd>0$ in the exponential case, we obtain
\eqref{eq:onestep} and this completes the proof of \Cref{thm:onestep}.
\end{proof}

\appendix

\section{$\kappa_1=\mathrm{var}(\logZ)/4$: A proof}
\label{sec:kappa1=var}

Recall from \eqref{eq:kappa_1_formula} that $\kappa_1= (1/2)\int \log (1 + e^{-x}) \nu(\dd x)$.
\medskip

\begin{proposition}
\label{th:kappa1=var}
Under the assumptions of Theorem~\ref{thm:stationary_asymptotics}, but requiring that $\xi>4$,  we have that
$\int_\bbR \log(1+ e^{-x})\nu (\dd x)= \bbE[ \logZ^2]/2$. 
\end{proposition}

\medskip

\medskip

\begin{proof}
For conciseness we treat only the case of assumption {\bf (T-1)} of \Cref{thm:stationary_asymptotics} and we remind the reader that the assumptions
include $\logZ_-\in \bbL^2$, so $\logZ\in \bbL^2$.  Note that these assumptions are weaker than those of our main result, \cref{th:main}.

We recall that \cref{thm:stationary_asymptotics} allows us to choose $\nu$ such that
\begin{equation} 
\label{eq:stationary_asymptotics2}
\nu((-\infty, x])\,=\,   x+ d_\nu + r_\nu(x)\ \text{ with }
r_\nu(x)\,=\begin{cases} O\left(x^{-(\xi-3)}\right) & \text{ for } x \to \infty\\
O(\vert x\vert) & \text{ for } x \to -\infty,
\end{cases}
\end{equation}
with a proper choice of  $d_\nu\in \bbR$. The (very weak) bound for $x\to-\infty$ comes from the boundedness of $F_\nu(x)=\nu((-\infty, x])$ in this limit. In fact, we have $F_\nu(x)=O(|x|^{-2})$ for $x\to -\infty$, as a consequence of
$\logZ_-\in \bbL^2$ and \cref{th:Y}. We remark that $F_\nu(x)=O(|x|^{-2})$ for $x \to -\infty$  and $\nu((-\infty, x])\sim x$ for $x \to \infty$ yield that  $\int_\bbR \log(1+ e^{-x})\nu (\dd x) = \int_\bbR \frac{1}{1+e^{x}} F_\nu(\dd x)$ is finite.

By \eqref{eq:fornu} applied to $g(x)= x \ind_{(-\infty, b]}(x)$ we have that for every $b \in \bbR$
\begin{equation}
\int x  \ind_{(-\infty, b]}(x) \, \nu(\dd x) = \iint \left(z+h(x)\right) \ind_{(-\infty, b]}(z+h(x)) \nu(\dd x) \zeta(\dd z)\,.
\end{equation}
Both sides diverge when $b\to \infty$, so we rearrange to obtain the form $I_1=I_2-I_3$, with:
\begin{equation}
\begin{aligned}
\label{eq:3terms}
I_1 & := \iint   \left(h(x)-x\right) \ind_{(-\infty, b]}(z+h(x))  \, \nu(\dd x) \zeta (\dd z)\, , \\
I_2 & := - \iint  z \, \ind_{(-\infty, b]}(z+h(x)) \nu(\dd x) \zeta (\dd z) \, , \\
I_3 & :=  \iint  \left( \ind_{(-\infty, b]}(z+h(x)) - \ind_{(-\infty, b]}(x) \right) x\, \nu(\dd x) \zeta (\dd z) 
 \, .
\end{aligned}
\end{equation}
We pass to the $b\to \infty$ limit in the three terms $I_1$, $I_2$ and $I_3$ separately. In all three cases we exploit the properties of $h(\cdot)$, in particular that
$h(x)=x+ O(\exp(-x))$ for $x \to \infty$, to replace $h(x)$ with $x$ in the argument of the indicator function. 

For $I_1$ we have
\begin{equation}
\begin{split}
I_1\, &=\, \iint   \log \left(1+e^{-x}\right) \ind_{(-\infty, b]}(z+h(x))  \, \nu(\dd x) \zeta (\dd z)
\\
&=\,
 \int_{(-\infty,b)}\int_{(-\infty,h^{-1}(b-z)]} \log \left(1+e^{-x}\right) \, \nu(\dd x) \zeta (\dd z)\\
 &=\, \bbP(\logZ<b) \int_\bbR  \log \left(1+e^{-x}\right) \nu(\dd x) - \int_{(-\infty, b)}\int_{(h^{-1}(b-z), \infty)} 
 \log \left(1+e^{-x}\right) \, \nu(\dd x) \zeta (\dd z)\,,
 \end{split}
\end{equation}
so, recalling that $\int_\bbR \log(1+ e^{-x})\nu (\dd x)< \infty$,
 \begin{equation}
 \left \vert I_1 -\int_\bbR \log(1+ e^{-x})\nu (\dd x) \right\vert \,=\, O \left(b^{-\xi}\right)
 + \int_{(-\infty, b)}\int_{(h^{-1}(b-z), \infty)} 
 \log \left(1+e^{-x}\right) \, \nu(\dd x) \zeta (\dd z)
 \,  .
 \label{eq:I1_firsttry}
 \end{equation}
The last integral is  bounded from above, separating the case $z\le b/2$ and the case $z\in ( b/2, b)$, by
 \begin{equation}
  \int_{(h^{-1}(b/2), \infty)} 
 e^{-x} \, \nu(\dd x) + \bbP\left(\logZ \in (b/2, b)\right)
  \int_\bbR \log \left(1+e^{-x}\right)\, \nu(\dd x) \,,
 \end{equation}
 and using \cref{eq:stationary_asymptotics2} we see that the first addendum is $O(\exp(-b/3))$ and the second is $O (b^{-\xi})$. 
 Therefore returning to \cref{eq:I1_firsttry}
\begin{equation}
\label{eq:forI1}
I_1\, =\, \int_\bbR \log(1+ e^{-x})\nu (\dd x) + O(b^{-\xi})\, .
\end{equation}

    For the term $I_2$, we start by observing that, since $h(x) > x$ and $h(\R) = (0,\infty)$,
\begin{equation}
\label{eq:2nd-diff}
\int \left( \ind_{(-\infty, b]}(z+x)- \ind_{(-\infty, b]}(z+h(x))\right) \nu(\dd x) \, =\,\begin{cases}\nu \left( \left( h^{-1}(b-z), b-z\right]\right) & \text{ if } z<b,
\\ 
\nu \left( \left( -\infty, b-z\right]\right) & \text{ if } z\ge b, 
\end{cases}
\end{equation}
which by  is $O(b^{-(\xi-3)})$ if $z< b/2$ and $O(b)$ otherwise. 
Therefore 
we see that the expression in 
\eqref{eq:2nd-diff} multiplied by $z$ and integrated with respect to $\zeta( \dd z)$ is $O(b^{-(\xi-3)})$ and so
\begin{equation}
I_2=
- \int z \, \nu \left( \left( -\infty, b-z\right]\right) \zeta (\dd z) + O(b^{-(\xi-3)})=
\bbE[ \logZ^2] - \int z r_\nu(b-z)\zeta(\dd z) + O(b^{-(\xi-3)}),
\end{equation}
where in the second step we used 
\eqref{eq:stationary_asymptotics2}. By separating once again the cases $z<b/2$ and $z\ge b/2$ we see  that
 $\int z r_\nu(b-z)\zeta(\dd z)= O(b^{-(\xi-3)})$. Therefore 
\begin{equation}
I_2\, =\,  \bbE[ \logZ^2]+ O(b^{-(\xi-3)})\, .
\end{equation}

Before handling the estimation of $I_3$, we establish the following estimate: 
\begin{equation}
\label{eq:for3limits}
\int_{(-\infty, y]} x\, \nu(\dd x) \, =\,\frac 12 y^2+C_\nu+ R_\nu(y), \quad  \text{ with } R_{\nu}(y)
= \begin{cases}  O\left(y^{-(\xi-4)}\right) & \text{ for } y \to \infty\,,
\\ O(y^2) & \text{ for } y \to -\infty\,,
\end{cases}
\end{equation}
with $C_\nu$ a real constant that is implicitly defined by \eqref{eq:for3limits} and will be given explicitly below.
In fact the integration by parts formula for measures yields 
\begin{equation}
\label{eq:byparts}
\int_{(-\infty, y]} x\, \nu(\dd x) \, =\, y F_\nu(y)- \int_{-\infty}^y F_\nu(x) \dd x\, ,
\end{equation}
because the identity function $x\mapsto x$ is differentiable, so by \eqref{eq:stationary_asymptotics2} we obtain for 
$y>0$ 
\begin{equation}
\begin{aligned}
\int_{(-\infty, y]} x\, \nu(\dd x) \, &=\, y F_\nu(y) -\int_0^y F_\nu (x) \dd x
-\int_{-\infty} ^0 F_\nu(x) \dd x 
\\
&=\,
\frac {y^2 }2 +y r_\nu(y)+ \int_y^\infty r_\nu(x) \dd x +\left[-\int_{-\infty} ^0 F_\nu(x) \dd x -\int_0^\infty r_\nu(x)  \dd x \right]\, .
\end{aligned}
\end{equation}
Considering the asymptotics for $y \to \infty$ we directly obtain the desired estimate, with $C_\nu$ equal to the expression in square brackets in the last line; note that the terms with $d_\nu$ cancel.
We then complete the proof of 
 \eqref{eq:for3limits} by observing that $F_\nu(x)=O(1/x^2)$ for $x \to- \infty$ yields  $R_\nu(y) \sim -y^2/2$ for $y \to -\infty$.

We now estimate $I_3$. We start by establishing that 
 for a suitable choice of $c>0$, for all $b \ge 1$
\begin{equation}\label{eq:forjsgdj}
\left \vert  \int \left[\ind_{(-\infty, b]}(z+x)-\ind_{(-\infty, b]}(z+h(x))   \right] x\, \nu(\dd x)\right \vert  \le c  \begin{cases}
 b^{-(\xi-4)} & \text{ if } z<b/2\, ,
 \\
 b^2  &\text{ if } b/2 \le z < b \, , \\
 1 &\text{ if }  z \ge  b\,. 
 \end{cases}
\end{equation}
If $z \ge b$ we have that $\ind_{(-\infty, b]}(z+h(x)) =0$ for every $x$ because $h(x)>0$ so 
the term in \eqref{eq:forjsgdj} is bounded by  
$ \int _{(-\infty, b-z]} \vert x\vert \, \nu(\dd x)\le \int _{(-\infty, 0]} \vert x\vert \, \nu(\dd x)< \infty$.  
If instead $z<b$ we have
\begin{equation}
\label{eq:jsgdj}
\begin{aligned}
0 & \,\le \, \int \left[\ind_{(-\infty, b]}(z+x)-\ind_{(-\infty, b]}(z+h(x))   \right] x\, \nu(\dd x) \,\\
 & =\,
 \int _{(h^{-1}(b-z), b-z]} x\, \nu(\dd x)\,\\
 & \leq (b-z) \nu \left( (h^{-1}(b-z), b-z] \right)
\end{aligned}
\end{equation} 
We estimate this last term by considering separately $z<b/2$ and $z \in [b/2, b)$.
In the first case $h^{-1}(b-z)= b-z +O(\exp(-b/2))$ so, by \eqref{eq:stationary_asymptotics2}, 
we see that $\nu \left( (h^{-1}(b-z), b-z]\right)=O(b^{-(\xi-3)})$. In the other case, $z \in [b/2, b)$, we use instead 
$\nu \left( (h^{-1}(b-z), b-z]\right) \le \nu \left( (-\infty, b/2]\right)=O(b)$. 
This yields \eqref{eq:forjsgdj}.

From \eqref{eq:forjsgdj} we easily infer that
\begin{equation}
\label{eq:int-bound5t}
 \iint   \left[\ind_{(-\infty, b]}(z+x)-\ind_{(-\infty, b]}(z+h(x))   \right] x\, \nu(\dd x) \zeta(\dd z)
 \, =\, O \left( b^{-(\xi-4)}\right)\, ,
    \end{equation}
    so  to estimate $I_3$ it suffices to consider 
    \begin{multline}
    \label{eq:forA1}
\iint  \left( \ind_{(-\infty, b]}(z+x)- \ind_{(-\infty, b]}(x) \right) x\, \nu(\dd x) \zeta (\dd z)\stackrel{\eqref{eq:for3limits}}=
\\
\int \left(\frac 12 \left( (b-z)^2-b^2\right)+ R_\nu(b-z) -R_\nu(b) \right)\zeta (\dd z) \, = \, \frac 12 \bbE[ \logZ^2]+ O \left(b^{-(\xi-4)}\right)\,,
    \end{multline}
    where the estimate of the term containing $R_\nu(b-z)$ is performed as before, by separately considering two cases:  for $z> b/2$, we use the fact that  $|R_\nu(b-z)|=O(z^2)$,   obtaining a bound $O(b^{-(\xi-2)})$; whereas for $z \le  b/2$, we obtain $O(b^{-(\xi-4)})$. Therefore 
\begin{equation}
I_3\, =\, \frac{1}{2} \bbE[ \logZ^2]+ O \left(b^{-(\xi-4)}\right)\, .
\end{equation}

By putting together the estimates on $I_1$, $I_2$ and $I_3$ we complete the proof of \Cref{th:kappa1=var}.
\end{proof}

\section*{Acknowledgements and funding}
{We are very grateful to an anonymous referee for several comments that helped improving our work and for pointing out to us the exact solution in Remark~\ref{rem:exact}.}
The work of R.L.G.\ was supported in part by the MIUR Excellence Department Project MatMod@TOV awarded to the Department of Mathematics, University of Rome Tor Vergata, CUP E83C23000330006. G.G.\ acknowledges the support of the Cariparo Foundation. Y.H.\ acknowledges the support of ANR LOCAL.

\section*{Declarations}
The authors have no competing interests to declare that are relevant to the content of this article.

\end{document}